\documentclass[12pt,oneside,english,shortlabels]{amsart}

\usepackage[T1]{fontenc}
\usepackage[latin9]{inputenc}
\usepackage{color}
\usepackage{babel}
\usepackage{amstext}
\usepackage{amssymb}
\usepackage{esint}
\usepackage[all]{xy}
\usepackage[unicode=true,pdfusetitle,
 bookmarks=true,bookmarksnumbered=false,bookmarksopen=false,
 breaklinks=false,pdfborder={0 0 0},backref=false,colorlinks=true]
 {hyperref}
\hypersetup{
 linkcolor=blue}
\usepackage{breakurl}

\makeatletter

\usepackage{amsthm}
\usepackage{enumitem}		% customizable list environments
\usepackage{mathrsfs} %for the \mathscr
\providecommand{\noopsort}[1]{}

\def\deg{\operatorname{deg}}

\let\cal\mathcal

\def\11{{\mathbf 1}}

\def\CC{{\mathbb C}}

\def\NN{{\mathbb N}}

\def\QQ{{\mathbb Q}}
\def\RR{{\mathbb R}}

\def\ZZ{{\mathbb Z}}

\newtheorem{thm}[subsection]{Theorem}

\newtheorem{lemm}[subsection]{Lemma}
\newtheorem{cor}[subsection]{Corollary}
\newtheorem{prop}[subsection]{Proposition}
\newtheorem{prr}[subsection]{Proposition}

\theoremstyle{definition}
\newtheorem{defn}[subsection]{Definition}
\newtheorem{example}[subsection]{Example}

\newtheorem{def-prop}[subsubsection]{Proposition-Definition}
\newtheorem{def-thm}[subsubsection]{Theorem-Definition}
\newtheorem{def-lem}[subsubsection]{Lemma-Definition}

\theoremstyle{remark}
\newtheorem{remark}[subsection]{Remark}

\theoremstyle{remark}

\theoremstyle{notation}
\newtheorem{notation}[subsection]{Notation}

\theoremstyle{plain}

\numberwithin{equation}{subsection}

  {\par\medskip\noindent #1\par\begingroup%
    \advance\leftskip by 1em\advance\rightskip by 1em}%
  {\par\endgroup}

\newcommand{\abs}[1]{\lvert#1\rvert}

\author{Georges Comte}
\address{Laboratoire de Math\'ematiques de l'Universit\'e de Savoie Mont~Blanc, UMR CNRS 5127,
B\^atiment Chablais, Campus scientifique,
73376 Le Bourget-du-Lac cedex, France}
\email{georges.comte@univ-savoie.fr}
\urladdr{http://gcomte.perso.math.cnrs.fr/}

\author{Chris Miller}
\address{Department of Mathematics
The Ohio State University
231 W. 18th Avenue
Columbus, OH 43210}
\email{miller@math.osu.edu}
\urladdr{https://people.math.osu.edu/miller.1987/}

\title[Points of bounded height on oscillatory sets]
{Points of bounded height on oscillatory sets}

\begin{document}

\begin{abstract}
We show that transcendental curves
in $\mathbb R^n$ (not necessarily compact) have few rational points of bounded height provided that the curves
are well behaved with respect to algebraic sets in a certain sense
and can be parametrized by functions belonging to a specified algebra of infinitely differentiable functions.
%The method is based on that of Bombieri and Pila.
Examples of such curves include logarithmic spirals and solutions to Euler equations $x^2y''+xy'+cy=0$ with $c>0$.
%%, and extends to the non-o-minimal setting some earlier results by others on Pfaffian curves and special functions.
\end{abstract}

%\dedicatory{\today}

\thanks{This work was initiated while the authors were in residence at the Mathematical Sciences Research Institute (Berkeley) during the Model Theory, Arithmetic Geometry and Number Theory program (January 20, 2014 to May 23, 2014).
Research by Comte also supported by ANR STAAVF. Research by Miller also supported by Universit\'e Savoie Mont Blanc and NSF Grant DMS-1001176.}

%DOES NOT SEEM TO WORK WITH sfmart
\subjclass[2010]{Primary 11G50; Secondary 03C64}

%MY STANDARD STATUS DISCLAIMER
%\thanks
%{This document should be regarded as preliminary.
%It has not yet
%Some version of it has
%been submitted for publication.
%Comments are welcome.}

\maketitle

\section{Introduction}\label{S:intro}
The {\bf height} of $(a_1/b_1,\dots,a_n/b_n)\in \QQ^n$ is by definition
$\max\{\vert a_i\vert,  \vert b_i \vert : i=1, \cdots, n \},$
when each pair $(a_i, b_i)\in \ZZ^2$ is coprime.
For $X\subseteq \RR^n$ and $T\geq 0$, let
$\#X(\QQ,T)$
%$h(X,T)$
be the number of points in $X\cap \QQ^n$ of height at most $T$.

The overarching question: What is the asymptotic behaviour of $\#X(\QQ,T)$ as $T\to +\infty$?
As it can be notoriously difficult to determine whether $X$ has any rational points at all, emphasis tends to be on the establishment of upper bounds.
We say that $\#X(\QQ,T)$ is sub-polynomial if
$$
\lim_{T\to+\infty} \frac{\log (\max\{1,\#X(\QQ,T)\})}{\log T}=0.
$$
For temporary convenience in exposition, we put
$$
\rho(X)=\varlimsup_{T\to+\infty}\frac{\log (\max\{1,\#X(\QQ,T)\})}{\log\log T}\in[0,+\infty],
$$
and say that $X$ has finite order if $\rho(X)<\infty$.

There is a natural split between the algebraic and the transcendental;
we are concerned with the latter.
Indeed, if $X$ is algebraic, then $\#X(\QQ,T)$ is usually not sub-polynomial (see for instance Browning, Heath-Brown and Salberger \cite{BrHeSa06}).
On the other hand, if $X$ is the graph of a  transcendental analytic function on a compact interval of $\RR$, then
$\#X(\QQ,T)$ is sub-polynomial by Bombieri and Pila~\cite{BoPi1989}, and this is sharp (see Pila \cite{Pi06counterexample}, or Surroca \cite{Su02}, \cite{Su06}).
More generally,
by Pila and Wilke~\cite{PiWi2006}, the sub-polynomial bound holds for the ``transcendental part'' of $X$ (see Section~\ref{S;logic} for the definition) if $X$ generates an o-minimal structure on the real field (see, e.g., van den Dries and
Miller~\cite{DrMi96} for the definition).
%van den Dries~\cite{Dr98}
Several results and conjectures even assert that $\rho(X)<+\infty$ if
$X$ has some specific dimension or is definable
in some specific o-minimal structures (see Binyamini and Novikov \cite{BiNo16},
Cluckers, Pila and Wilkie \cite{ClPiWi16},  Jones, Miller and Thomas \cite{JoMiTh11}, Jones and Thomas \cite{JoTh12}, Pila \cite{Pi2006}, \cite{Pi07}, \cite{Pi10}).
In this direction, and for more connections to logic, see Section~\ref{S;logic}.

In this paper, we do not restrict attention to non-oscillatory sets (that is, sets that generate an o-minimal structure).
Rather, we develop some sufficiency conditions for showing that certain kinds of curves in $\RR^n$, oscillatory or not, have finite order (see also
Besson \cite{Be11}, \cite{Be14}, Boxall and Jones \cite{BoJo15b}, \cite{BoJo15}, Masser \cite{Ma11}).
%Throughout, we let $X$ be a parametrized curve.
%, especially uniform upper bounds for families.
Our main result is technical; a full statement is best postponed (see
Theorem~\ref{fewratpoints}), but we shall give some examples now as motivation.

For $a>0$, we say that a $C^\infty$ function $g\colon [a,\infty[\to\RR$ is  \textbf{slow}
(see Definition \ref{slowp})
if there exist nonnegative real numbers $A$, $B$, $C$ and $D$ such that,
for all $x\geq a$ and $p\in\NN$,
$$
\lvert \frac{g^{(p)}(x)}{p!}\rvert\leq   D (Ap^B\frac{\log^Cx}{x})^p.
$$
We denote by $S([a,+\infty[)$ the set of slow functions (with possibly different data $A,B,C,D$).
As observed in Remark \ref{re:AlgeraSlow}, the set of slow functions is an algebra stable under derivation.

Let $\mathcal E$ be the collection of all elementary real functions of one variable, as defined in Khovanskii~\cite[\S1.5]{Kh91}, whose domain contains an unbounded-above open interval.
All functions in $\mathcal E$ are (real-)analytic.

%%%%%%%%%%%%%%%%%%%%%%%%%%%%%%%%%%%%%%
\begin{prop}\label{introcurve}
Let $f,g ,s_1, s_2\in  ( \mathcal E\cap S([a,+\infty[))\cup \{\mathrm{Id} \}$
 and  $f,g$ belong to the set of functions $h$ satisfying
$$
\ h=\mathrm{Id} \   \hbox{ or } \ \exists \alpha \ge 0, \
\forall x\ge a, \forall p\ge 0,  \  \lvert h^{(p)}(x)\rvert \le \alpha^p.
$$
Let $F,G >0$, $\ell,q\in \NN^*$, and let $u,v$ be real numbers that are both rational numbers or such that one of them is irrational and is not a  $U$-number of degree $\nu=1$ in Mahler's classification
\footnote{ See Baker \cite[Chapter 8]{Ba75} for definitions and details, and note that almost all numbers are
irrational numbers that are not $U$-numbers of degree
$\nu=1$ by \cite[Theorem 8.2]{Ba75}.}.
Let $x_0$ be such that the curve
$$
X:=\{\bigl(
u+ e^{-Fx}f(s_1(x^{\ell})),
v+ e^{-Gx}g(s_2(x^q))\bigr): x\geq x_0\}
$$
is defined and contains no infinite semialgebraic set.
Then $X$ has finite order.
% \changeGC{(see Remark \ref{re:elementarycurve})}.
\end{prop}
%
%[MAYBE SOME EXAMPLES HERE OF EXOTIC $G_1,G_2$?]
%
%%%%%%%%%%%%%%%%%%%%%%%%%%%%%%%%%%%%

Example of such curves are easy to produce.

\begin{example}\label{exampleintro}
The curve given by the parametrization
$$ \left(\log2+\frac{\arctan \log^2 x}{x^5(2+\cos^3\log x)},
\pi+ \frac{\sin \log x}{\sqrt{x}(1+\log\log x)}
\right) , \ x\ge 2 $$
has finite order.
\end{example}

An explicit bound on $\rho$ can be
computed in terms of  the complexity of the elementary functions $f,g,s_1$ and $s_2$.
Thus, for particular choices of $f,g,s_1, s_2$,
%after computation of the complexity of these elementary functions,
we can get more refined  results, such as:

%%%%%%%%%%%%%%%%%%%%%%%%%%%%%%%%%%
\begin{prop}\label{introspiral}
Let $a_1,a_2,F,G,c_1,c_2$ be nonzero real numbers
with $F,G>0$.
Let $\ell,q\in \NN^*$ and let $u,v$ be as in Proposition
\ref{introcurve}.
Then
$$
X:=\{\bigl(
u+a_1e^{Fx}\cos(c_1x^\ell),
v+a_2e^{Gx}\sin(c_2x^q)\bigr): x \in \RR\}
$$
has order at most $5+4\max\{\ell,q\}$.
\end{prop}
%\noindent
%(The bound is probably rather loose; all we know for certain is that there are choices of parameters yielding $\rho\geq 1$.

As a special case, we have $\rho(\mathbb S_\omega)\leq 9$  for
every logarithmic spiral
$$
\mathbb S_\omega:=\{\bigl(e^{x}\cos(\omega x), e^{x}\sin(\omega x)\bigr): x \in \RR\}\quad (\omega\neq 0).
$$
Observe that $\mathbb S_\omega$ is a maximal (real-time) trajectory of $\dot z=(1+i\omega)z$
and a subgroup of $(\mathbb C^*,\cdot\,)$.
It is easy to see that $\rho(\mathbb S_{\pi/\log 2})\geq 1$.

We have a related result for functions.

%%%%%%%%%%%%%%%%%%%%%%%%%%%%%%%%%%%%%%%%%%%%%%%%%%%%%%%%%%%%%%%%%%%%%%%%%%%%%%%
\begin{prop}\label{introeuler}
Let $f, s\in \cal ({\cal E}\cap S([a,+\infty[))\cup \{\mathrm{Id}\}$,  and $\ell$ and $f$ be as in Proposition \ref{introcurve}.
Then the graph, X, of  $f\circ s \circ \log^\ell$ has finite
order. In particular,
let $a$ and $c$ be nonzero real numbers,
$\ell \in \NN^*$ and $f\in\{\sin,\cos\}$.
Then the graph of
$af(c\log^\ell) $ has order at most $5+4\ell$.
\end{prop}

%%%%%%%%%%%%%%%%%%%%%%%%%%%%%%%%%%%%%%%%%%%%%%%%%%%%%%%%%%%%%%%%
%\changeGC{
%\begin{remark}

%Translations of any set $X$ considered above in Propositions
%\ref{introcurve}, \ref{introspiral} or \ref{introeuler},
%by vectors such that every coordinate is
%an algebraic number or a transcendental $S$-number in Mahler's classification  (almost all numbers are $S$-numbers), also have finite order
%(see Remarks \ref{re: on the limit point of the curve 2} (4) and  \ref{re: spiral around non rational centre}).

%\end{remark}
%}

With $\ell=1$, we obtain maximal solutions to the Euler equation $x^2y''+xy'+c^2y=0$, indeed, $\{\cos (c\log),\sin(c\log)\}$ is a set of fundamental solutions on $]0,\infty[$.
It is easy to see that $\rho\geq 1$ for $c=\pi/\log 2$.
The proof generalizes to show that all transcendental slow functions in $\mathcal E$ that satisfy certain further technical conditions have finite order (see Proposition \ref{fewratpointsgraph}).

To the best of our knowledge, these results yield the first known examples of real-analytic submanifolds (embedded and connected) of the plane having nonzero finite order and whose intersection with some
real-algebraic set has infinitely many connected components.
(By Lindemann-Weierstrass, $\rho(\sin x)=0$.
It is an easy exercise that $\rho(\sin(\pi x))=+\infty$.)

For $X$ as in Propositions~\ref{introcurve}, \ref{introspiral} and~\ref{introeuler}, it follows from work of Pila \cite{Pi2006} that
every compact subset of $X$ has finite order, but via proof that does not yield finite order for $X$ itself.
We remedy this by establishing the existence
of compact connected $K_X\subseteq X$ such that
$\rho(X\setminus K_X)<\infty$, thus yielding $\rho(X)=\max(\rho(K_X),\rho(X\setminus K_X))<+\infty$.
%%%%%%%%%%%%%%%%%%%%%%%%%%%%%%%%%%%%%%
%%%%%%%%%%%%%%%%%%%%%%%%%%%%%%%%%%%%%%
The constants witnessing the bounds appearing in \cite{Pi2006} depend a priori on the choice of the compact subset of $X$, and in no case could we conclude in some visibly simple way that the set of all such constants
may be bounded as the length of our parametrizing interval goes to infinity.
For instance, by \cite{Pi2006}, every bounded subset of the graph of $\sin \pi x$ has finite order, but again $\sin \pi x$ has infinite order on any unbounded interval.
Thus, dealing with the oscillatory case requires tight control of any involved constants appearing in any known methods as the parameter interval we are looking
at increases in length.
Since we use here the determinant method of \cite{BoPi1989}, as in \cite{PiWi2006}, and since the constants that arise from this method come from bounds on the derivatives
%of any order
of the functions
parametrizing the set, we impose some tameness on the derivatives via the algebra of slow functions.
This condition may be viewed as the
oscillatory and noncompact counterpart of Pila's mild parametrizations;
this somewhat answers the wish of~\cite[4.3 Remark 3]{Pi2006}.
 
One then obtains in Theorem \ref{fewratpoints} and Proposition \ref{fewratpointsgraph} general explicit
bounds for $\#X(\QQ,T)$ in a split form of a product of two factors, reflecting two features of different nature of $X$: one arising from the slowness
hypothesis; the other from how often our curve intersects algebraic curves of given degrees.

%%%%%%%%%%%%%%%%%%%%%%%%%%%%%%%%%%%%%%
%%%%%%%%%%%%%%%%%%%%%%%%%%%%%%%%%%%%%%

There are non-oscillatory unbounded functions of finite nonzero order; as examples, the function $2^x$ has order $1$ (indeed, the only rational points are $(k,2^k)$, $k\in\ZZ$), and
the restriction to the positive real line of the Euler gamma function $\Gamma$ has order at most $2$ (see \cite{BoJo15}) and at least $1$ (consider the values at positive integers and recall Stirling's formula).
The restriction to $]1,\infty[$ of the Riemann zeta function $\zeta$ has order at most $2$ (again see~\cite{BoJo15}), but it is not known if the order is positive.
Our methods are not confined to the oscillatory case; for illustrative purposes, we shall give alternate proofs of the finiteness of the orders of $\Gamma$ and $\zeta$ (see Sections \ref{RiemanZeta} and \ref{EulerGamma}), though our bounds are weaker than those already known.

\section{Counting rational points on curves.}\label{Scounting}
%%%%%%%%%%%%%%%%%%%%%%%%%%%%%%%%%%%%%%%%%%%%%%%%%%%%%%%%%%%%%%%%%%%%%

For $X\subseteq \RR^n$ and $T\ge 0$, we let $X(\QQ,T)$ denote the set of rational points in $X$ with height $\le T$.
%
% and by $\#X(\QQ,T)$ its cardinality
%(instead of the notation $h(X,T)$ used for convenience in the introduction where we discussed the asymptotics $\#X(\QQ,T)$). Namely, as already indicated in the introduction,
%$$X(\QQ,T)=\{(x_1, \cdots, x_n) \in X\cap \QQ^n ; \forall i=1, \cdots, n, \ \exists   (a_i,b_i)\in \ZZ\times \ZZ^* $$
%$$\ \hbox{such that }
%x_i=a_i/b_i \hbox{ and } \abs{a_i}\le T, \abs{b_i}\le T \}. $$
Let $\varGamma\subset \RR^n$ be a parameterized curve, that is,
the image of $n$ smooth (i.e., infinitely differentiable) functions $\gamma=(f_1, \cdots, f_n) \colon I \to \RR^n$, for $I$ an interval
in $\RR$,
such that $\gamma(I)=\varGamma$.

The goal of this section is to provide, under some hypothesis
on the derivatives of the $f_i$, a bound for $\#\varGamma(\QQ,T)$.
Evidently, it is equivalent to provide such a bound on the projection
of $\varGamma$ onto some coordinate plane of $\RR^n$; from now on, we can assume without loss of generality that our forthcoming assumptions on the $f_i$ concern
only two coordinate functions among the $f_i$, or more conveniently, that $n=2$. Therefore, in the sequel we let
$\gamma=(f,g)$ be a parametrization of a given curve $\varGamma\subset \RR^2$.

For $N, L\in \RR^*$ and an interval $J\subset \RR$, we put
$$\varGamma_{N,L}:=\gamma([N,N+L])$$
$$\varGamma_{N,+\infty}:=\gamma([N,+\infty[,$$
$$\varGamma_J:=\gamma(J). $$

%

 %%%%%%%%%%%%%%%%%%%%%%%%%%%%%%%%%%%%%%%%%%%%%%%%%%%%%%
 \begin{lemm}\label{BoundAboveDet}
 Fix $\mu \in \mathbb{N^*}$. Let $I$ be an interval of length $L$.
Let $x_1$, \dots, $x_{\mu}$ be points in $I$ and
$\psi_1, \cdots, \psi_\mu$ be  $C^{\mu-1}$-functions from $I$ to $\RR$.
Set $$\Delta = \det (\psi_i (x_j)).$$
Then
$$\vert \Delta \vert \le  L^{\frac{\mu(\mu-1)}{2}}
\sum_{ s\in S(0,\cdots, \mu-1)}
 \sigma_{1,s(0)} \sigma_{2,s(1)}\cdots \sigma_{\mu,s(\mu-1)},$$
 where  $S(0,\cdots, \mu-1)$ is the permutation group of $\{0,\cdots, \mu-1\}$ and
 $$\sigma_{i,p}=\sup_{\xi\in I} \frac{1}{p !}   \vert \psi_i^{(p)}(\xi)\vert , \
p \in \{ 0,\cdots, \mu-1 \}.$$
 \end{lemm}

 %%%%%%%%%%%%%%%%%%%%%%%%%%%%%%%%%%%%%%%%%%%%%%%%%%%%%%

%%%%%%%%%%%%%%%%%%%%%%%%%%%%%%%%%%%%%%%%%%%%%%%%%%%%%%%%%%%%%%%%%%%%%%%%%%%%%
\begin{proof}
We use the by-now classical bound for $\vert \Delta \vert$ (see \cite[Proposition $2$]{BoPi1989}, \cite{Pi2006}):
$$\vert \Delta \vert \le \vert  V(x_1, \cdots, x_\mu)  \vert
\cdot \vert \det \Big(\frac{\psi_j^{(i-1)}(\xi_{ij})}{(i-1)!} \Big)_{i,j=1, \cdots, \mu}  \vert,  $$
where $V$ is the Vandermonde determinant and $\xi_{ij}$ are points in $I$. The statement follows.

\end{proof}
%%%%%%%%%%%%%%%%%%%%%%%%%%%%%%%%%%%%%%%%%%%%%%%%%%%%%%%%%%%%%%%%%%%%%%
%%%%%%%%%%%%%%%%%%%%%%%%%% Debut du calcul complet
%%%%%%%%%%%%%%%%%%%%%%%%%%
In order to bound from above the determinant $\Delta$
we consider from now on  a specific bound condition on the derivatives of our parametrization $\gamma$.
This condition will allow for control of the maximal length $L$ of a parameter interval $[x,x+L]$ such that $\varGamma_{x,L}(\QQ,T)$ is contained in a single algebraic curve of given degree $d$ (see Proposition \ref{length}) and thus to eventually control the number of such algebraic curves needed to cover $\varGamma(\QQ,T)$ (see Proposition \ref{NumberIntervals}).

\begin{defn}\label{slowp}
Let $a>1$. We say that a smooth parametrization
$$\gamma=(f,g)\colon [a,+\infty[\to \RR^2$$
of a bounded curve $\varGamma$ is a {\bf slow
parametrization} of $\varGamma$ (with data $A,B,C, u$ and $b$) if:
\begin{enumerate}
\item there exist $u\in \RR$ and $b:[a,+\infty[\to \RR_+ $ decreasing to zero such that, for all $x\ge a$,
 $$ \vert f(x) -u \vert \le b(x);$$
\item there exist nonnegative real numbers $A,B,C$ such that for all $p\ge 0$ and $x\ge a$,
$$
\vert \frac{f^{(p)}(x)}{p!}  \vert \le \varphi_p(x),
\
\vert \frac{g^{(p)}(x)}{p!}  \vert \le \varphi_p(x),$$

where
$$\varphi_p(x)=\Big( Ap^B\frac{\log^{C}x}{x} \Big)^p.$$

\end{enumerate}
We call a function satisfying condition $(2)$ above a {\bf slow function}, and say that the set of constants $A,B,C$ is {\bf attached} to the function, or attached to
 the slow parametrization $\gamma$.
 \end{defn}

%%%%%%%%%%%%%%%%%%%%%%%%%%%%%%%%%%%%%%%%%%%%%%%%
\begin{remark}\label{re: f and g bounded by 1}
In full generality one should rather define $\varphi_p$ as
\begin{equation}\label{eq: full definition of slow algebra}
   \varphi_p(x)=D\Big( Ap^B\frac{\log^{C}x}{x} \Big)^p
\end{equation}
for some real number $D$ (as in the introduction) in order to have stability by addition for the set of slow functions, as pointed out in Remark \ref{re:AlgeraSlow} below.
But the condition
$f\le 1, g\le 1$, made for simplicity in forthcoming computations, is made without loss of generality, up to dividing $f$ and $g$ by an integer bounding $f$ and $g$, which is harmless for counting rational points.
\end{remark}

\begin{remark}\label{re: on the limit point of the curve 1}
In our Definition \ref{slowp} the first coordinate $f$ of a slow parametrization $\gamma$ goes to $u$ as the parameter goes to infinity, at speed at least $b$. In the sequel, for applications we will often consider the case
$b(x)=1/x^E$, with $E>0$, and
the Diophantine properties of $u$ have to be considered
for counting rational points on $\varGamma$; see Remark \ref{re: on the limit point of the curve 2} $(4)$.
\end{remark}

\begin{remark}\label{1/f}
If a function $k$ (bounded from above or not) is greater than $1$ and satisfies
condition $(2)$ of Definition \ref{slowp} for $p\ge 1$, then $1/k$ is slow.
\end{remark}

\begin{remark}\label{re:AlgeraSlow}
For $a>1$, let us denote by $S([a,+\infty[)$
the set of slow functions defined on
$[a,+\infty[$
(with the bounding function $\varphi_p$ as in (\ref{eq: full definition of slow algebra}),
and with possible different data $A,B,C,D, u, b$ for each of them).
An easy computation,
essentially based on the formula (\ref{derprod}) given below,  shows that  $S([a,+\infty[)$ is a subalgebra  of
$C^\infty([a,+\infty[)$, stable under derivation.
\end{remark}

%%%%%%%%%%%%%%%%%%%%%%%%%%%%%%%%%%%%%%%%%
\begin{remark}\label{constantssmallorbig}
For a given slow parametrization, we can always assume that $A,B,C$ are large enough,  up to quantitatively weakening the condition of being a slow parametrization.
The condition that the parameter has to be larger than $a$  in this definition is technical. On one hand we are essentially interested in the part of the curve corresponding to a neighbourhood of infinity of the parameter, because this part of the curve is the oscillatory part when the parametrization is analytic.
On the other hand, for analytic curves ${\varGamma}$ restricted to a compact interval of parameters, since we have a mild parametrization  (see \cite{Pi2006}), one gets a bound
for $\#{\varGamma}(\QQ,T)$ as given in Theorem \ref{fewratpoints}, depending on how much $\Gamma$ cuts algebraic curves of given degree, or at least we have a general sub-polynomial bound for $\#{\varGamma}(\QQ,T)$,
provided by \cite{PiWi2006} for the o-minimal context.
\end{remark}

%%%%%%%%%%%%%%
\begin{remark}
The bound $(2)$ in Definition \ref{slowp} may be seen as the unbounded version of the mild para\-me\-tri\-za\-tion introduced in~\cite{Pi2006}. If an analytic function $g \colon \RR_+\to \RR$ satisfies the bound $(2)$ of Definition \ref{slowp}
with $B=0$,
then there is a complex analytic continuation of $g$ to the half plane $\{x\in \CC, \Re(z)>0\}$, since the radius of convergence of the Taylor series of $g$ at $x$ is greater than
$\displaystyle \frac{x}{A\log^C x}$.
\end{remark}
\begin{remark}
A slow parametrization may also be thought of as the oscillatory version of Yomdin-Gromov parametrization for o-minimal sets (see Gromov~\cite{Gr86}, Yomdin~\cite{Yo87} and~\cite{PiWi2006}), since a slow parametrization will provide, for fixed $T$ and $d$, a finite number $n(T)$ of intervals $I$ such that $ \varGamma_I(\QQ,T)$ is contained in a single algebraic curve of degree $d$.
In the o-minimal case, since the number $n(T)$ given by a Yomdin-Gromov parametrization is not controlled, one cannot expect
better bounds for $\#\varGamma(\QQ,T)$ than $\alpha_{\epsilon} T^{\epsilon}$ (for any $\epsilon>0$).
Contrariwise, here,  the fast decay   on the derivatives imposed by a slow parametrization will provide a sufficiently small number $n(T)$ (Lemma \ref{NumberIntervals}) that will in turn imply that $\varGamma$ has finite order,  provided that good enough B\'ezout bounds (Definition~\ref{de: transcendental and Bezout bound}) are available (Theorem \ref{fewratpoints}).
%, like does a mild parametrization in the compact case (see \cite{JoMiTh11},   \cite{Pi2006}).
It is worth noting that, starting with a given slow parametrization, the constants $\alpha$ and $\beta$ in our bound $\alpha \log^\beta T$ for $\#\varGamma(\QQ,T)$ in any of Theorem \ref{fewratpoints},  Propositions \ref{fewratpointsLogSpiral} or \ref{fewratpointsgraph} will depend explicitly on the data $A,B,C$ and $b$, though the density of rational points in the set $\varGamma$ should be independent of any choice of parametrization. This obvious constraint theoretically explains  why, in practice, the range of slow parametrizations, and reparametrizations that improve the decay of a given slow parametrization or the decay of $b$, is necessarily limited. When one composes a slow parametrization with a smooth increasing bijection
$\sigma \colon [a',+\infty[\to [a,+\infty[$ having smaller and smaller derivatives as $x$ goes to infinity (e.g., $x\mapsto \log x$), the
decay of the derivatives of the reparametrization should improve. But at the same time, since $\sigma$ is contracting the distance and the slow parametrization we started with has smaller and smaller derivatives as $x$ gets larger and larger, this effect is counterbalanced.  It thus appears that any given slow parametrization is a quite stable and optimal form of parametrization with respect to reparametrization and counting rational points of bounded height in a curve.
\end{remark}

We assume in what follows that $\gamma=(f,g)$ is a slow parametrization of the curve $\varGamma$.
Now let us choose $d\in \NN^*$ and let us denote $(\alpha_1,\alpha_2)\in \NN^2$  by $\alpha$ and
 the function $f^{\alpha_1}g^{\alpha_2}$ by $\gamma^\alpha$.

In the following lemmas, we will use the classical formula
\begin{equation}\label{derprod}
\frac{(f^\alpha g^\beta)^{(p)}}{p!}
=\sum_{i_1+ \cdots i_\alpha +j_1+\cdots + j_\beta =p}
\frac{ f^{(i_1)}\cdots f^{(i_\alpha)}g^{(j_1)}\cdots g^{(j_\beta)}}{i_1 ! \cdots i_\alpha ! j_1 ! \cdots j_\beta !},\quad \alpha,\beta,p\in\NN.
\end{equation}

%%%%%%%%%%%%%%%%%%%%%%%%%%%%%%%%%%%%%%%%%%%%%%%%%%%%%%%%%%%%%%%%%%%%%%%%%%%%%%%%%%%
\begin{lemm}\label{BoundAboveDer} Let $\gamma$
be a slow parametrization of a plane curve
$\varGamma$,  with constants $A,B,C$.
Then for any $\alpha\in \NN^2$ and $p\in \NN$,
\begin{equation*}\frac{\vert (\gamma^\alpha)^{(p)}(x)\ \vert}{p!} \le (p+1)^{\alpha_1+\alpha_2}
A^{p} p^{Bp} \frac{\log^{Cp}x}{x^p}.\qedhere
\end{equation*}
%$$\hbox{ and } \vert (\phi^\alpha)^{(p)}(x)\ \vert \le \frac{ C(\alpha,p) }{ x^{\alpha_1 +p} }, \  \hbox{ for } \alpha_2 =0, $$
%where $C(\alpha, p)$ is a constant depending only on $\alpha, p$
%and $C(\alpha,\ell, p)$ a constant depending only on $\alpha, \ell, p$.
\end{lemm}

(Hence, $\gamma^\alpha \in S([a,+\infty[)$.)
%%%%%%%%%%%%%%%%%%%%%%%%%%%%%%%%%%%%%%%%%%%%%%%%%%%%%%%%%%%%%%%%%%%%%%%%%%%%%%%%%%%

%%%%%%%%%%%%%%%%%%%%%%%%%%%%%%%%%%%%%%%%%%%%%%%%%%%%%%%%%%%%%%%%%%%%%%%%%%%%%%%%%%%
\begin{proof}
By formula (\ref{derprod}),
$$
\frac{\vert (\gamma^\alpha)^{(p)}(x) \vert}{p!}
\le \sum_{i_1+ \cdots i_{\alpha_1} +j_1+\cdots + j_{\alpha_2} =p}
A^p {i_1}^{B i_1} \cdots {i_{\alpha_1}}^{B i_{\alpha_1}}
  {j_1}^{B j_1} \cdots  {j_{\alpha_2}}^{B j_{\alpha_2}} \frac{\log^{Cp} x}{x^p}$$
$$
\le  (p+1)^{\alpha_1+\alpha_2}
A^{p} p^{Bp} \frac{\log^{Cp} x}{x^p}.$$

\end{proof}
%%%%%%%%%%%%%%%%%%%%%%%%%%%%%%%%%%%%%%%%%%%%%%%%%%%%%%%%%%%%%%%%%%%%%%%%%%%%%%%%%%%
We now recall another bound that will be used in the proof of Proposition \ref{BoundAboveGeneral}.
%%%%%%%%%%%%%%%%%
\begin{lemm}\label{majPi2006} For any $\mu \ge 1$ and $B>0$,
$$\prod_{r=0}^{\mu-1} r^{B r}\le \mu^{B\frac{\mu(\mu-1)}{2}}.$$
\end{lemm}
%%%%%%%%%%%%%%%%%%

%%%%%%%%%%%%%%%%%
\begin{proof}
We use the classical bound for such a product (see also \cite[Proposition 2.2]{Pi2006}).
We have
\begin{align*}
\log(\prod_{r=0}^{\mu-1} r^{Br})
&=B \sum_{r=1}^{\mu-1} r\log r \le B\int_{1}^{\mu}t\log t \ dt\\
&\le B\frac{\mu(\mu-1)\log \mu}{2}.\qedhere
\end{align*}
% ***For $k\in \NN, 1 \le k\le p$, for $j= \lfloor  \frac{p}{k+1}\rfloor+1, \cdots, \lfloor  \frac{p}{k}\rfloor$, we have $j(k+1)>p$. It follows that for $ j= \lfloor  \frac{p}{k+1}\rfloor+1, \cdots, \lfloor  \frac{p}{k}\rfloor$,  integers $m_j=k+1, \cdots, p$ are not allowed when $\sum_{j=1}^p jm_j=p$, leaving only $k+1$ possibilities for these $ \lfloor  \frac{p}{k}\rfloor-  \lfloor  \frac{p}{k+1}\rfloor$ integers $m_j$. It follows that $$ M_p\le \prod_{k=1}^{p} (k+1)^{\frac{p}{k(k+1)} +1}= (p+1)!  \prod_{k=1}^{p} (k+1)^{\frac{p}{k(k+1)}}. $$
% Now
%$$ \log( \prod_{k=1}^{p} (k+1)^{\frac{p}{k(k+1)}}) = p\sum_{k=1}^p  \frac{\log(k+1)}{k(k+1)}.  $$ But since $x\mapsto \frac{\log(x)}{x(x-1)}$ increases for $x\ge 2$, we have $$\sum_{k=1}^p  \frac{\log(k+1)}{k(k+1)} \le  \int_2^{p+2} \frac{\log(x)}{x(x-1)} \ dx = $$
\end{proof}
%%%%%%%%%%%%%%%%%
Let us now fix some integer $d\ge 1$ and set
$$\mu=\frac{(d+2)(d+1)}{2}=\# \{\alpha \in \mathbb{N}^2 : \abs{\alpha}  \leq d \}, \ \ \rho=\frac{\mu(\mu-1)}{2},$$
 and for
$\mu$ points $x_1,\cdots, x_\mu$ given in the domain of definition of $\gamma$,
consider the determinant
$$\Delta=\det(\gamma^\alpha(x_j))_{\alpha\in \Delta_2(d), j\in \{1, \cdots, \mu\}}.$$

%%%%%%%%%%%%%%%%%%%%%%%%%%%%%%%%%%%%%%%%%%%%%%%%%%%%%%%%%%%%%%%%%%%%%%%%%%%%%%%%%%
\begin{prr}\label{BoundAboveGeneral}
 Let $\gamma$
be a slow parametrization of a plane curve
$\varGamma$.
With the above notation,
for any $d\ge 1$ there exist $N=N(C)\ge 1 $  and $C(d,A,B)>0$
such that, for any $L>0$ and $x_1, \cdots, x_\mu\in [N; N+L]$,
$$\vert \Delta \vert \le C(d,A,B)L^\rho \frac{\log^{ C\rho} N}
{N^ \rho}.$$
Furthermore, one can take
$C(d,A,B)=\mu! A^\rho \mu^{d(\mu-1)} \mu^{B\rho}$.
% and $N(d,1)$ independent of $d$ and equal to $1$.
\end{prr}
%%%%%%%%%%%%%%%%%%%%%%%%%%%%%%%%%%%%%%%%%%%%%%%%%%%%%%%%%%%%%%%%%%%%%%%%%%%%%%%%%%

%%%%%%%%%%%%%%%%%%%%%%%%%%%%%%%%%%%%%%%%%%%%%%%%%%%%%%%%%%%%%%%%%%%%%%%%%%%%%%%%%%
\begin{proof}
Let us first note that
for any $p=0,\cdots, \mu-1$,
 $\log^{C p}x/x^{p}$ is decreasing on $[N,+\infty[$ with, for instance  $N=N(C)=e^C$.
 %Let us also note that $$\sum_{(\alpha,\beta)\in \NN^2; \alpha+\beta\le d} \alpha \ge \sum_{j=1}^d j(d+1-j)\ge \sum_{j=1}^d j  \ge d^2/2 .$$
 Combining Lemma \ref{BoundAboveDet}, Lemma \ref{BoundAboveDer} and these two remarks, one obtains for $x_1,\cdots, x_\mu\in [N,N+L]$
%%%%%%%%%%%%%%%%%%%%%%%%%%%%%%%%%%%%%%%%%%%%%%
$$\vert \Delta \vert
\le   L^{\frac{\mu(\mu-1)}{2}}  \sum_{ s\in S(0,\cdots, \mu-1) }
A^{\sum_{j=0}^{\mu-1} j}
\prod_{r=0}^{\mu-1} r^{Br}
\prod_{r=0}^{\mu-1} (r+1)^d
\frac{\log^{ C\sum_{j=0}^{\mu-1} j} N }{ N^{\sum_{j=0}^{\mu-1}j}}. $$
%%%%%%%%%%%%%%%%%%%%%%%%%%%%%%%%%%%%%%%%%%%%%%
$$\le A^\rho \mu^{d(\mu-1)} L^{\rho}
\frac{\log^{ C\rho} N}{ N^\rho}
\sum_{ s\in S(0,\cdots, \mu-1) }
\prod_{r=0}^{\mu-1} r^{Br}.
$$
Now using Lemma \ref{majPi2006}, we have
\[
\vert \Delta \vert
\le \mu! A^\rho \mu^{d(\mu-1)} \mu^{B\rho}L^{\rho}
\frac{\log^{ C\rho}N}{N^\rho}.\qedhere
\]
% Partie qui traite le cas où la borne log ne marcje pour l>1 et où donc il convient d'obtenir classiquement une borne en T^\epsilon pour l>1
% Finally, in the case $\ell>1$, we may choose $N=N(d,\ell)$ big enough such that $$\frac{\log^{(\ell-1)\rho}(N)}{N^{\frac{d^2}{6} }}\le 1,  $since $\log^{(\ell-1)\rho}\sim_{N \to +\infty} \log^{(\ell-1)\frac{d^4}{4}}$.
%And in the case $\ell=1, $ for  $N\in [1,+\infty[$ one has $$\vert \Delta \vert\le (\mu-1)!  \mu^{d(\mu-1)} \mu^{2\rho}\frac{L^{\rho}}{ N^{\frac{d^2}{2} + \rho}}.$$
\end{proof}
%%%%%%%%%%%%%%%%%%%%%%%%%%%%%%%%%%%%%%%%%%%%%%%%%%%%%%%%%%%%%%%%%%%%%%%%%%%%%%%%%%
Let us now fix $T\ge 1$ and
  assume that the points $x_1, \cdots, x_\mu\in [N,N+L]$, with the notation of Proposition \ref{BoundAboveGeneral}, are such
that $(f(x_j),g(x_j))$, $j=1, \cdots, \mu$, is a pair of rational points of height $\le T$. In this case, if $\ \Delta\not=0$ we have,
$$ \vert  \Delta \vert
\prod_{j=1, \cdots, \mu } \vert a_j^db_j^d\vert  \ge 1,$$
 where $a_j$ and $b_j$ are the denominator of $f(x_j)$ and $g(x_j)$.
But since
$$\prod_{  j=1, \cdots, \mu } \vert a_j^d  b_j^d \vert \le T^{2d\mu }   ,$$
we also have, again if $\ \Delta\not=0$,
$$T^{2d\mu } \vert   \Delta \vert \ge 1.$$
Now, considering Proposition \ref{BoundAboveGeneral}, as soon as
%%%%%%%%%%%%%%%%%%%%%%%%%%%%%%%%%%%%%%%%%%%%%%%%%%%%%%%%%%%%%%%%%%%%%%%%%%%
\begin{equation}
   T^{2d\mu}C(d,A,B)L^\rho \frac{\log^{ C\rho} N}{N^\rho} <1
\label{VanishingCondition}
\end{equation}
%%%%%%%%%%%%%%%%%%%%%%%%%%%%%%%%%%%%%%%%%%%%%%%%%%%%%%%%%%%%%%%%%%%%%%%%%%%
 we necessarily have $\Delta=0$. Let us fix $L$ as
%%%%%%%%%%%%%%%%%%%%
\begin{equation}
L:=L(d,A,B,T,N)=C'(d,A,B) \frac{N}{\log^C N} T^{\frac{-4d}{\mu-1}},
\label{ChoiceOfL}
\end{equation}
%%%%%%%%%%%%%%%%%%%%
with $C'(d,A,B)=C(d,A,B)^{\frac{-1}{\rho}}$.
It follows that for any  interval $[N;N+L[$, with $N\ge N(C)$  given by Proposition \ref{BoundAboveGeneral},
and  $L$ given by equation (\ref{ChoiceOfL}), for any choice of points $x_1, \cdots, x_\mu \in [N;N+L[$
and such that $\gamma(x_j)\in \varGamma_{N,L}(\QQ,T)$, the rank of $(\gamma^\alpha(x_j))_{\alpha\in \Delta_2(d), j=1,\cdots, \mu, }$ is
$\le\mu-1$.

We now proceed similarly to  Lemma 1 of \cite{BoPi1989}.
 Let $a$ be the maximal
rank of $(\gamma^\alpha(x_j))$ over all $x_1,\cdots, x_\mu
\in [N;N+L[$ such that  $\gamma(x_j)\in \varGamma_{N,L}(\QQ,T)$ and
let $M=(\gamma^\alpha(x_j))_{\alpha\in I, j\in \{1, \cdots, a\}}$
be of rank $a$, for some fixed $x_1,\cdots, x_a\in [N;N+L[$
 such that  $\gamma(x_j)\in \varGamma_{N,L}(\QQ,T)$ and some  $I\subset \{\alpha\in \NN^2, \vert \alpha \vert\le d\}$ of cardinality $a$.
Since $a<\mu$, we can choose $\beta=(\beta_1,\beta_2)\in \Delta_2(d)\setminus I$.

Let us denote by $f(y)\in \RR[y]$ the determinant of the matrix
$$(M' y^\delta)_{\delta\in I\cup \{\beta\}},$$ where
$y=(y_1, y_2)$,  $ y^\delta=y_1^{\delta_1}y_2^{\delta_2}$,
and $M'$ is $M$ augmented
by the line $\gamma^\beta(x_j)_{j\in \{1, \cdots, a\}}$.
Then $f(y)$ is a polynomial with real coefficients that is not zero, since the coefficient of the monomial
$y^\beta$ in $f(y)$ is
the nonzero minor $\det(M)$, and the degree of $f(y)$ is at most $d$.
Furthermore for any $(x_1,x_2)\in \varGamma_{N,L}(\QQ,T)$ we have
$f(x_1,x_2)=0$, by definition of the maximal rank $a$.

Note that in case there are fewer than $\mu$ points $x$ in $[N,N+L]$
such that $\gamma(x)\in \varGamma_{N,L}(\QQ,T)$, those points $\gamma(x)$
are certainly in some algebraic curve of degree less than $d$, since the dimension of the space of polynomials of $\RR[y]$ of degree at most $d$ is $\mu$.

To sum up this discussion, we have proved the following statement.
%%%%%%%%%%%%%%%%%%%%%%%%%%%%%%%%%%%%%%%%%%%%%%%%%%%%%%%%%%
\begin{prr}\label{length}
 Let $\gamma$
be a slow parametrization of a plane curve
$\varGamma$.
Let $d\in \NN^*$ and $T\ge 1$. There exists $N(C)\ge 1$ such that for any $N\ge N(C)$ exists  $L=L(d,A,B,T,N)$, such that the set
  $\varGamma_{N,L}(\QQ,T)$
is contained in an algebraic curve of $\RR^2$ of degree at most  $d$.  Furthermore one can take
$L=C'(d,A,B)\frac{ N}{\log^C N}T^{-\nu}$, with
 $\nu=\frac{4d}{\mu-1}$ and
$C'(d,A,B)=A^{-1}\mu^{-B} \mu^{\frac{-2d}{\mu}} (\mu!)^{\frac{-2}{\mu(\mu-1)}}$.
\end{prr}
%%%%%%%%%%%%%%%%%
So far, we have introduced the notion of slow parametrization and given bounds for such parametrization, like in Proposition \ref{length}. This is in order to eventually bound the number of rational points of bounded height in the range $\varGamma\subset \RR^2$ of  the slow parametrization we consider.
At this point, one may have the naive idea that the slower this parametrization is (that is the smaller $A,B$ and $C$ are), the better this bound should be, since the slower our parametrization is, the larger the length $L$ provided by Proposition  \ref{length} is, and thus the smaller the number $n(T)$ of intervals $I$ such that $\varGamma_I(\QQ,T)$ is contained in one algebraic curve of given degree is.
But of course the density in $\varGamma$ of rational points of given height does not depend on the parametrization of $\varGamma$,
that maybe could be taken slower than it is, up to a convenient reparametrization of the half line by itself. There is no contradiction here. In fact such a reparametrization
also increases the time needed to parametrize the rational points
of height less then $T$, and this growth balances the gain obtained on the
bounds of the derivatives after reparametrization.
Nevertheless, we believe that this balancing effect must be made visible in our final bound, with the hope that one can, in a particular situation, optimize this balance. For this goal one introduces a control of the speed at which a slow parametrization, or a suitable change of variables made to obtain a slower parametrization, runs through the rational points of given bounded height. On the other hand, so far, we have not mentioned that in order
to obtain a bound on $\#\varGamma(\QQ,T)$ from assumptions on a parametrization of $\varGamma$, one has to exclude some noninjective
behaviour of the parametrization under consideration: passing too many times by the same rational point leaves no possibility to bound
$\#\varGamma(\QQ,T)$ throughout a parametrization.
The next definition proposes such a notion of quantitative control on the time needed to go through all the points of given bounded height as well providing a quantitative control of, let us say, the noninjectivity of our parametrization. The better is this control, called a height control function, the better will be our final bounds. It is thus worth considering as a separate additional assumption such a control function and state our bounds in terms of this given height control function.

%%%%%%%%%%%%%%%%%%%%%%%%%%%%%%%%%%%%%
\begin{defn}\label{defheightcontrolfunction}
Let $\gamma \colon [a,+\infty[\to \RR^2$
be a parametrization of a plane curve
$\varGamma$. We say that a function $\varphi \colon [a,+\infty[\to \RR$ is a {\bf height control function}  for $\gamma$ if, for all $T\ge 1$,
$$\gamma^{-1}(\varGamma(\QQ,T))\subset [a,\varphi(T)].$$
\end{defn}
%%%%%%%%%%%%%%%%%%%%%%%%%%%%%%%%%%%%%

\begin{remark}\label{hcf}
In the following particular cases, one can easily compute a height control function of $\gamma=(f,g)$. The preliminary observation is useful: for $a,p\in \QQ$, $b,q\in \QQ^*$,  $a/b-p/q =0$ or $\vert a/b-p/q \vert \ge K/\vert q \vert$, for $K=1/\vert b \vert $.
\begin{enumerate}
\item[(1)]\label{rk:hcf1}
In case that $f$ and $g$
are decreasing  respectively to $u,v\in \QQ$, a height control function of $\gamma$
is given by $\varphi(T)=\min( f^{-1}(\frac{K}{T}),g^{-1}(\frac{K}{T}) )$, for some $K>0$.
\item[(2)]\label{rk:hcf2}
If $u,v\in \QQ$ and $\vert f -u \vert$ and $\vert g -v \vert$ are respectively
bounded from above by functions $b$ and $c$ that decrease to $0$,  a height control function of $\gamma$
is given by $\varphi(T)=\max( b^{-1}(\frac{K}{T}),c^{-1}(\frac{K}{T}) )$, for some $K>0$, as soon as $f-u$ and $g-v$ have no common zero on $[a,+\infty[$ (note that contrariwise to
$(1)$, in this case the functions $f-u$ or $g-v$ may   have zeroes between
$\min( b^{-1}(\frac{K}{T}),c^{-1}(\frac{K}{T}) )$
and $\max( b^{-1}(\frac{K}{T}),c^{-1}(\frac{K}{T}) )$).
\item[(3)]\label{rk:hcf3}
In case $u \in \QQ$, any decreasing  function $b \colon [a,+\infty[\to \RR$ going to $0$ and
bounding  $\vert f-u \vert$
%(resp. $g$)
gives rise to a height control function of $\gamma$ defined by $\varphi(T)=b^{-1}(\frac{K}{T})$, for some $K>0$, as soon as $f$
%(resp. $g$)
does not take the value $u$.
In particular,
when $\gamma$ is a slow parametrization as in Definition
\ref{slowp}, with $u\in \QQ$ and $b(x)=1/x^E$, $E>0$, one can take $\varphi(T)=T^{\frac{K}{E}}$, for some $K>0$, as soon as $f$ does not take the value $u$ on $[a,+\infty[$.

\item[(4)]\label{re: on the limit point of the curve 2}
% Let us consider the case that the slow parametrization $\gamma$ has its first coordinate $f$ going to $u\not=0$ as the parameter goes to infinity. As already mentioned in Remark \ref{re: on the limit point of the curve 1}, in this case one has to consider the Diophantine properties of $u$. The case $u\in \QQ$ is easy and can be treated directly by translation by $u$, therefore we can assume that $u \not\in \QQ$.
In case $u\not \in \QQ$, let us denote by $ \tau $ an {\bf irrationality measure function} of $u$, that is a function such that  for any $p,q\in \NN$ with height less than $T$, $\vert u-\frac{p}{q} \vert \ge e^{-\tau(T)}$.
When $(f,g)$ is slow, since $\vert f(x)-u \vert \le b(x) $ and $b$
decreases to $0$, a height control function for $\gamma$ is given by $b^{-1}(e^{-\tau(T)})$. For instance by Roth's Theorem \cite{Ro55}, in case $u$ is an algebraic number, $\tau$ may be $\log \frac{T^3}{C}$, where $C>0$ depends on $u$.  In case $u=\pi$ (by Cijsouw's Theorem \cite{Ci77}), or in case $u=\log w$, for $w$ an algebraic number not $0$ or $1$,  $\tau$ may be $K\log T$ (see \cite[Chapter 3]{Ba75}).
In fact $\tau(T)=K\log T$, and thus $\varphi(T)=b^{-1}(1/T^K)$, is possible for any number which is not a $U$-number of degree $\nu=1$ in Mahler's classification, and almost all numbers being $S$-numbers, almost all numbers are not $U$-numbers (see \cite[Chapter 8]{Ba75}).

\end{enumerate}
\end{remark}

For fixed $T\ge 1, d\in \NN^*$ and  with the notation introduced in Proposition \ref{length},
we define a sequence $(x_n)_{n\in \NN}$ of real numbers by
$$x_0=N(C), \ x_{n+1}=x_n+C'(d,A,B)   \frac{x_n}{\log^C x_n}   T^{-\nu}. $$
Then $[x_n,x_{n+1}]$ is an interval in $[N(C),+\infty[$ of length $$L_n=C'(d,A,B) \frac{x_n}{\log^C x_n}T^{-\nu}$$ such that $\varGamma_{x_n,L_n}(\QQ,T)$
is contained in one algebraic curve of degree $\le d$.
%
%

%%%%%%%%%%%%%%%%%%%%%%%%%%%%%%%%%%%%%

\begin{remark} For $\gamma$ a slow parametrization of a curve $\varGamma$, with height control function
$\varphi$ and for
$T\ge 1$ fixed, since the sequence $(x_n)_{n\in \NN}$ goes to
$+\infty$, we can  cover the interval $[N(C),\varphi(T)]$, whose image contains the rational points of
$\varGamma$ of height $\le T$,  with a finite number of intervals
$[x_n,x_{n+1}]$. An upper bound on this number provides an upper bound for the number of curves of degree $\le d$ containing $\varGamma_{N(C),+\infty}(\QQ,T)$. The following Lemma
gives such an upper bound.

\end{remark}
%%%%%%%%%%%%%%%%%%%%%%%%%%%%%%%%%%%%%

%Following Remark \ref{constantssmallorbig}, we assume from now on that $E\le 1$ in our given slow parametrization. This technical condition is convenient to obtain the following bounds in $T$, for every
%$T\ge 1$. Without this assumption, our bounds would be valid, but only for $T$ bigger than some constant depending on $E$.

%%%%%%%%%%%%%%%%%%%%%%%%%%%%%%%%%%%%%
\begin{lemm}\label{NumberIntervals} Let
$\gamma$ be a slow parametrization of a plane curve $\varGamma$, with height control function
$\varphi$ and let
$T\ge 1$ and $d\in \NN^*$. We denote by $n(T)$ the least $n\in \NN$ such that $x_n\ge \varphi(T)$.
Then

$$
n(T)\le  \frac{T^{\nu}
\log^{C+1} \varphi(T)}{\log(2) \min(1, C'(d,A,B))}+1,
$$
where $C'(d,A,B)$ and $\nu$ are given by Proposition \ref{length}.
  In particular, we can cover  $[N(C); \varphi(T)]$ with at most
  $$\Big\lfloor \frac{T^\nu
   \log^{C+1} \varphi(T)}{\log(2) \min(1, C'(d,A,B))}\Big\rfloor +1 $$
   intervals $I$
such that $\varGamma_I(\QQ,T)$ is contained in one algebraic curve of $\RR^2$ of degree $\le d$.
\end{lemm}
%%%%%%%%%%%%%%%%%%%%%%%%%%%%%%%%%%%%%

\begin{proof} %We first notice that the function $]1,+\infty[\ni x\mapsto \frac{x^{\epsilon}}{\log^C x}$ is bounded from below on $]1,+\infty[$ by $\frac{e^C} {C^C} \epsilon^C:=a_C \epsilon^C$.
%Since $N(C)\ge 1$, we have $x_n\ge 1$, for any $n\in \NN$.
 In case $x_0\le x_1\le \cdots \le x_{n-1}\le \varphi(T)$ we
 have
 $$x_n\ge x_0(1+C'(d,A,B)   \frac{T^{-\nu}}{\log^C\varphi(T)}    )^n.$$
  Since by definition $x_{n(T)-1}\le \varphi(T) $, we have
  $$ \varphi(T) \ge x_{n(T)-1}\ge x_0(1+C'(d,A,B)  \frac{T^{-\nu}}{\log^C\varphi(T)})^{n(T)-1}.$$
 In particular
 $$
 n(T)\le \frac{\log \varphi(T)-\log x_0}{
 \log(1+C'(d,A,B)T^{-\nu}/\log^C\varphi(T))}+1.$$
 %And since $\epsilon^C T^{-\nu}=(E\frac{d^2}{2\rho})^C T^{-\nu}\le 1$, having assumed $E\le 1$.
% $$
 %a_C C'(d,A,B)= \frac{e^C}{C^C} \cdot
 %A^{-1}\mu^{-B} \mu^{\frac{-2d}{\mu}}(\mu!)^{\frac{-2}{\mu(\mu-1)}}\le 1,$$
 %because we can assume in any case  that $C>e$,
 %
 %
 %
 In case $C'(d,A,B) T^{-\nu}/\log^C\varphi(T)  \ge 1 $, one has
$$n(T)\le  \frac{\log \varphi(T)}{\log 2}+1 \le \frac{T^\nu
\log^{C+1} \varphi(T)}{\log 2 } +1.$$
The last bound in this double inequality is harmless, since we will see in Theorem \ref{fewratpoints} that the term $T^\nu$ is a constant for a good choice of $d$ as a function of
$T$.
In case $C'(d,A,B) T^{-\nu}/\log^C\varphi(T) \le 1$,
 by concavity of the $\log$ function we obtain
\begin{equation*}
n(T)
 \le  \frac{T^\nu \log^{C+1} \varphi(T)}{
 \log(2) C'(d,A,B)}+1.\qedhere
\end{equation*}
\end{proof}

%\begin{remark}\label{boundforNslow}
%By Remark \ref{hcf}, for $\gamma$ a slow parametrization as defined
%in Definition \ref{slowp}, we can always take
%\changeGC{
%$$ n(T) \le  \frac{T^\nu \log(T)}
% {E\log(2)\min(1, a_CC'(d,A,B)\epsilon^C)}. $$}
%\end{remark}

%%%%%%%%%%%%%%%%%%%%%%%%%%%%%%%%%%%%%%%%%%%%%%%%
\begin{defn}\label{de: transcendental and Bezout bound} We say that a curve $\varGamma\subset \RR^2$ is {\bf transcendental} if it contains no infinite semialgebraic set.
A {\bf B\'ezout bound for $\varGamma$} is any quantity ${\cal B}(x,d)$ that dominates
the number of points of $\varGamma_{[a,x]}\cap P^{-1}(0)$ as
$P$ ranges over all nonzero polynomials in $\RR[X,Y]$ of degree at most $d\in\NN^*$.
\end{defn}

%%%%%%%%%%%%%%%%
\begin{notation} For $x\in\RR$, put $\log_+ x=\max(1,\log x)$.
\end{notation}

%%%%%%%%%%%%%%%%%%%%%%%%%%%%%%%%%%%%%%%%%%%%%%%%
\begin{thm}\label{fewratpoints} Let $\gamma $ be a slow parametrization of a transcendental
plane curve $\varGamma$, with height control function
$\varphi$, and let $T\ge 1$. Then there exists a constant $\alpha=\alpha(A,B,C)$
such that
$$\# \varGamma_{N(C),+\infty}(\QQ,T)\le
\alpha
  \log_+^{2(B+C)}( T)
\log^{C+1}(\varphi(T)){\cal B}(\varphi(T),\log_+ T),$$
where $N(C)$ is given by Proposition \ref{BoundAboveGeneral}, and can for instance be $e^C$.

In particular $\varGamma$ has finite order as soon as $\varphi$ is polynomially bounded and
${\cal B}(x,d) \le Q(\log x, d)$, for some polynomial $Q$, or as soon as $\varphi(T)$ is bounded by a power of $\log(T)$ and ${\cal B}(x,d) \le Q(x, d)$, for some polynomial $Q$.
\end{thm}
%In particular
% $$\# \varGamma_{[N(C,E),+\infty[}(\QQ,T)\le \frac{\alpha}{E} \changeGC{\log_+^{2(B+C)+1}(T)} b(T^{\frac{1}{E}},\changeGC{\log_+(T)}), $$ where $N(C,E)$ is given by Proposition \ref{BoundAboveGeneral}, and $N(C,E)$ can for instance be $e^\frac{C}{E}$.

%%%%%%%%%%%%%%%%%%%%%%%%%%%%%%%%%%%%%%%%%%%%%%%%
\begin{proof}
By
Lemma \ref{NumberIntervals}, the numbers of intervals $I$ we need to cover $[N(C), \varphi(T)]$, in such a way that
$\varGamma_I(\QQ,T)$ is contained in one algebraic curve of degree at most $d$, is less than
$$ \frac{T^\nu \log^{C+1} \varphi(T)}{\log(2)\min(1,C'(d,A,B))} +1,
 $$
with
$$\frac{1}{C'(d,A,B)}=
A\mu^B \mu^{2d/\mu}(\mu!)^{2/(\mu(\mu-1)}.$$
Thus, $\#\varGamma_{N(C,E),+\infty}(\QQ,T)$ is bounded by
$$\left( \frac{T^\nu  \log^{C+1} \varphi(T)}{\log(2)\min(1,C'(d,A,B))}+1 \right){\cal B}(\varphi(T),d).$$
By Proposition \ref{length}, $\nu=\displaystyle\frac{4d}{\mu-1}\le
\frac{8d}{(d+1)(d+2)-2}\le \frac{8}{d}$.
Taking
$$d=  \lfloor\log_+ T\rfloor,$$ we find for instance $T^{\nu}\le e^{16}$.
%Let us now bound $\frac{1}{\epsilon^C}$ in term of $\log_+ T$.
% Since   $1/\epsilon=2\rho/Ed^2$, we have
%$$1/\epsilon^C\le \frac{4^C}{E^C}(1+\log_+ T)^{2C}\le
% \frac{4^C}{E^C} (1+1/\log 2)^{2C}
%\log^{2C}_+ T.$$
Finally, let us bound each factor of $1/C'(d,A,B)$.
We have, since
$\frac{d^2}{2}\le \mu\le 2d^2$,
$$ \mu^B\le 2^B(1+ \log_+ T)^{2B}\le 2^B (1+1/\log 2)^{2B} \log^{2B}_+ T.$$
$$ \mu^{\frac{2d}{\mu}}=e^{\frac{4}{d}\log(2d^2)}\le e^{12}.$$
$$ (\mu!)^{\frac{2}{\mu(\mu-1)}}
\le \mu^{\frac{2}{(\mu-1)}}
\le e^{\frac{4 \log \mu}{\mu}}
\le e^{\frac{4}{e}}.$$
Now in case $\min(1,C'(d,A,B))=1$, we have that the number of intervals $I$ we need to cover $[N(C), \varphi(T)]$ is bounded by
$ \alpha' \log^{C+1}\varphi(T), $
for some real number $\alpha'$.

 On the other hand, if $\min(1,C'(d,A,B))\neq 1 $, we have obtained  as a bound $ \alpha'' \log^{2(B+C)}_+ T
\log^{C+1}\varphi (T)$, for some real number $\alpha''$ depending on our data $A,B,C$. In any case, one can take for a bound
$\alpha \log_+^{2(B+C)}T
\log^{C+1} \varphi(T)$, with $\alpha=\max(\alpha', \alpha'')$.
%The last sentence of the proposition  comes from Remark \ref{boundforNslow}.

Assuming that  $\varphi$ is polynomially bounded and
${\cal B}(x,d) \le P(\log x, d)$, for some polynomial $P\in \RR[X,Y]$, one directly obtains that
$\varGamma_{N(C),+\infty}$ has finite order. But since ${\cal B}(x,d) \le P(\log x, d)$, one  also
deduces that $\varGamma_{[a, N(C)]}$ has finite order, by the same computation on $[a,N(C)]$,
that reduces in the compact situation to the computation of \cite{Pi2006} for mild parametrization.

\end{proof}

%%%%%%%%%%%%%%%%%%%%%%
\begin{remark}
The bound given for $\#\varGamma_{N(C),+\infty}(\QQ,T)$ in Theorem \ref{fewratpoints} is a product of two factors, one coming from the fast decay of the derivatives of the slow parametrization, the other one,
${\cal B}(\varphi(T), \log_+ T)$, depending on how much the curve intersects algebraic curves
of degree $\lfloor \log_+ T \rfloor $ on the parameter interval $[N(C),\varphi(T)]$.
Those two factors may behave independently.
We know, by  \cite{Pi06counterexample}, \cite{Su02} or \cite{Su06}, of functions analytic on a neighbourhood of a compact interval having asymptotically as many as possible rational
points of height less than $T$  in their graph (for instance, for $\epsilon \in ]0,1[$, more than $\frac{1}{2} e^{2\log^{1-\epsilon} T}$ points, for an infinite sequence of values of $T$, whereas it has to be less than $C_\epsilon e^{\epsilon\log T}$ by \cite{BoPi1989} or \cite{PiWi2006}). Since the graph of a function analytic on a neighbourhood of a compact interval automatically inherits a mild
parametrization, that is the compact version of our slow parametrization, it means that for those  functions the second factor in our bound is as big as possible, while the first one stays bounded by a power of $\log T$. On the other hand for some families of curves $\varGamma$, one knows that $\varGamma$
cuts algebraic curves of given degree in few points with respect to this degree, and in this case, the split form of the bound of $\#\varGamma_{N(C),+\infty}(\QQ,T)$ in Theorem \ref{fewratpoints} provides a small value.
\end{remark}
%%%%%%%%%%%%%%%%%%%%%%%%%%%%%%%%%%%%%%%%%%%%%%%%%%%%%%
\begin{remark}\label{rk: slow_+ parametrizations}
Parametrizations $\gamma$ with stronger decay than the decay $\varphi_p$ of smooth pa\-ra\-me\-tri\-za\-tions of Definition \ref{slowp} provide better bounds for $\#\varGamma(\QQ,T)$. For instance, for some positive real number $E$ and with the notation of Definition \ref{slowp}, one can consider
parametrizations $\gamma=(f,g)$ such that $g$ is slow and $f$ satisfies for
all $p\ge 0$
\begin{equation}\label{eq: slow_+ parametrization}
\vert \frac{(f-u)^{(p)}}{p!}(x) \vert \le \frac{1}{x^E}\varphi_p(x).
\end{equation}
In this case
%%%%%%%%%%%%%%%%%%%%%
\begin{itemize}
\item  by Remark \ref{hcf}$(3)$, when $u\in \QQ$ one can take $T^{\frac{K}{E}}$ as a height control function  for $\gamma$ if $f-u$ has no zeros; by Remark
\ref{hcf}$(4)$, in case $u\not\in \QQ$ has
an irrationality measure function $\tau(T)$ of the form $K\log T$, $K>0$ (in particular when $u$ is not a $U$-number of degree $1$),
then
one can take $T^{\frac{K}{E}}$ as height control function for $\gamma$.
%%%%%%%%%%%%%%%%%%%%%%%
\item
Since the decay of the derivatives of
$f$ is improved by condition (\ref{eq: slow_+ parametrization}), comparing to the case where $f$ is simply slow, assuming condition (\ref{eq: slow_+ parametrization}), the same computations lead
 in Lemma \ref{BoundAboveDer} to a denominator $x^{p+E\alpha_1}$ instead of $x^p$, in Proposition \ref{BoundAboveGeneral} to a denominator
$\displaystyle N^{\rho+E\frac{d^2}{2}}$ instead of $N^\rho$, and
$N(C,E)=e^{\frac{C}{E}}$ instead of $N(C)=e^C$. As
a consequence, in Theorem \ref{fewratpoints} we obtain a smaller factor
$\log\varphi(T)$ instead of $\log^{C+1}\varphi(T)$.
\end{itemize}
\end{remark}

%%%%%%%%%%%%%%%%%%%%%%%%%%%%%%%%%%%%%%%%%%%%%%%%%%%%%%%%%%%%%%%%%%%%%
\section{ Some Examples. }
%%%%%%%%%%%%%%%%%%%%%%%%%%%%%%%%%%%%%%%%%%%%%%%%%%%%%%%%%%%%%%%%%%%%

We apply in this section Theorem \ref{fewratpoints}, the main statement of Section \ref{Scounting},
  which provides a bound for the number of rational points of prescribed height in a curve
with slow parametrization $\gamma$ and convenient height control function $\varphi$ as soon as we know a convenient B\'ezout bound ${\cal B}(x,d)$ for such a curve.
In particular, we indicate how to obtain Proposition \ref{introcurve}
by proving in detail a special case of Proposition \ref{introspiral}, and similarly for Proposition \ref{introeuler}.

%%%%%%%%%%%%%%%%%%%%%%%%%%%%%%%%%%
\subsection{Spirals.}

One typical family of oscillatory  curves that Theorem \ref{fewratpoints} allows us to treat is  the following family of ``fast'' logarithmic spirals.
Let $\ell,q\in \NN^*$, $F, G >0$
%$\in\NN^*?$
and let

\begin{equation}\label{eq:ParametrizationSpiral}
\phi_\ell:=\frac{1}{x^F}\sin \circ \log^\ell(x),\ \
\psi_q:=\frac{1}{x^G}\cos \circ \log^q(x),
\end{equation}
defined  on some unbounded interval $[a,+\infty[$ in $\RR_+^*$, with $a>1$.
%We assume that $F\le G$, without loss of generality.
%We can also assume that that $F\le G=1$, in order to meet, in the forthcoming Lemma \ref{lem:slowparaforspiral}, the condition $E\le 1$ imposed in the previous section from Lemma \ref{NumberIntervals}. This can be done
%because the change of variable
%$y=x^{\frac{1}{G}}$ does not affect the form of our parametrization; it only changes coefficients $H$ and $K$.
%Note that, for instance for $\ell=q=F=G=1$,
Put $\gamma_{\ell,q}:=(\phi_\ell, \psi_q)$.
%is a parametrization of a trajectory $\varGamma$
%of the linear vector field $v(x,y):=(y-x,-y-x)$ (see \cite{Bi15} for bounds on the number of  rational points counted on a compact subset of a trajectory of certain vector fields).
%
Observe that the image $\varGamma$ of $\gamma_{\ell,q}$ is the same as that of
$$
t\mapsto (e^{-Ft}\sin^d t, e^{-Gt}\cos ^d t),\quad t\geq \log a.
$$
The following two Lemmas \ref{majdkLogl} and \ref{majdkSinLogl} will be used to prove in Lemma \ref{slowspiral} that $\gamma_{\ell,q}$ is a slow parametrization of the curve $\varGamma$.

%%%%%%%%%%%%
\begin{lemm}\label{majdkLogl}
For $\ell\ge 1$,
$p\ge 1$ and $x\ge e$,
$$   \frac{(\log^\ell)^{(p)} x}{p!}  \le  2^\ell  p^\ell \frac{\log^{\ell-1} x }{x^p} .$$
\end{lemm}

%%%%%%%%%%%%%
\begin{proof} We use formula (\ref{derprod}) with $f=1$, $g=\log$
and $\beta=\ell$. We have
$$  \frac{(\log^\ell)^{(p)} x}{p!}
=\sum_{j_1+\cdots + j_\ell =p}
\frac{g^{(j_1)}(x)\cdots g^{(j_\ell)}(x)}{  j_1 ! \cdots j_\ell !}.$$
Let $k$ be the number of nonzero indices $j_r$ in the
term  $\frac{g^{(j_1)}(x)\cdots g^{(j_\ell)}(x)}{  j_1 ! \cdots j_\ell !}$; then this term is equal to
$$
\frac{\log^{\ell-k} x}{ x^p j_1 \cdots j_\ell },$$
where only the non zero $j_r$'s appear in the denominator.
We then have, since $\log x\ge 1$,
$$
\frac{(\log^\ell)^{(p)} x}{p!}
\le (p+1)^\ell \frac{\log^{\ell-1} x }{x^p}
\le  2^\ell p^\ell \frac{\log^{\ell-1} x }{x^p}.
$$

\end{proof}
In the next Lemma, we use the following classical formula

\begin{equation}\label{FaaDiBruno}
\frac{(f\circ g)^{(p)}(x)}{p!}
= \sum_{1m_1+2m_2+\cdots+pm_p=p}
\frac{f^{(m_1+\cdots+m_p)}(g(x)) \prod_{j=1}^p
(\frac{g^{(j)}(x)}{j!})^{m_j}}{m_1!\cdots m_p!}
\end{equation}

%%%%%%%%%%%%%%
\begin{lemm}\label{majdkSinLogl}
Let $d\in \RR$ and $f$ be a smooth function defined on $[d,+\infty[$.
Let
$\alpha \ge 1$ be such that $\vert f^{(p)}(x)\vert \le \alpha^p$
for all $x\ge d$ and $p\ge 0$.
\begin{enumerate}
\item[(1)]
If $s\in S([e,+\infty[)$ with range in
$[d,+\infty[$, then $f\circ s \in S([e,+\infty[)$.
\item[(2)]
If $\ell\ge 1$, $p\ge 0$ and $x\ge e$, then
$$
\frac{\vert  (f\circ \log^\ell)^{(p)}( x)\vert}{p!}
\le (\alpha 2^\ell)^p
  p^{(\ell+1) p} \frac{\log^{p(\ell-1)} x}{x^p}$$
(and so  $f\circ \log^\ell \in S([e,+\infty[)$).
\end{enumerate}
\end{lemm}

%%%%%%%%%%%%%%%
\begin{proof}
 (1)~Let us denote by $A$, $B$ and $ C$ the constants attached to the slow function $s$.
 Note that $f\circ s$ is bounded (by $1$), hence so is $f$. By formula (\ref{FaaDiBruno}), one has for $p\ge 1$
$$ \vert \frac{(f\circ s)^{(p)}(x)}{p!} \vert
\leq \sum_{1m_1+2m_2+\cdots+pm_p=p} \alpha^p
 \prod_{j=1}^p  A^{jm_j}j^{B jm_j} \frac{\log^{Cjm_j} x}{x^{jm_j}}$$
 $$
  \le \alpha^p A^p p^{(B+1) p}\frac{\log^{Cp} x}{x^p}    . $$

 (2)~One cannot use directly  statement $(1)$, since strictly speaking, $\log^\ell$ is not slow on $[e,+\infty[$, as it is not bounded, but nevertheless the  computation made to prove~$(1)$ does not use that $s$ is bounded and shows that, for any $p\ge 1$, one has the bound announced in $(2)$. Indeed by Lemma \ref{majdkLogl},  we can take the constants $A=2^\ell, B=\ell, C=\ell-1$ for $\displaystyle   (\log^\ell)^{(p)}/p!  $  to satisfy the bound of Definition \ref{slowp},  for any $p\ge 1$.
% ***  avec le meilleur bound du Lemme \ref{majdkLogl} $$\le \ell^p\frac{\log^{p(\ell-1)}(x)}{x^p} \sum_{1m_1+2m_2+\cdots+pm_p=p} \ell^pp^{(\ell-1) p} \le \ell^p p^{\ell p} \frac{\log^{p(\ell-1)}(x)}{x^p}.$$*** On peut encore incorporer un meilleur bound par le Lemme \ref{majjmj}.
\end{proof}
%%%%%%%%%%%%%%%%%%%%%%%%%%  Fin du calcul complet
%%%%%%%%%%%%%%%%%%%%%%%%%%

%%%%%%%%%%%%%%%%%%%%%%%%%
\begin{lemm}\label{slowspiral}
 Let  $d\in \RR$, $f$ and $g$ be two smooth functions defined on $[d,+\infty[$, and $F,G>0$.
Assume that $f$ and $g$ are in the set of functions $h$ satisfying
$$h=\operatorname{Id} \hbox{ or } \exists \alpha (\ge 1), \
\forall x\ge d, \ \forall p\ge 0,\ \vert h^{(p)}(x)\vert \le \alpha^p.$$
\noindent
Let $s$ and $\sigma$ be two slow functions on $[e,+\infty[$
with range in $[d,+\infty[$.
Then, for $x\ge e$, the parametrization
$$ x\mapsto (\frac{1}{x^F}f\circ s(x),
 \frac{1}{x^G}g\circ \sigma(x))$$
is a slow parametrization, satisfying condition
(\ref{eq: slow_+ parametrization}) of Remark \ref{rk: slow_+ parametrizations}.
In particular the parametrization $\gamma_{\ell,q}=(\phi_\ell, \psi_q)$ of (\ref{eq:ParametrizationSpiral}) is a slow parametrization
of the spiral $\gamma_{\ell,q}([e,+\infty[)$,
satisfying condition (\ref{eq: slow_+ parametrization}), with height control function
$\varphi(T)=T^{\frac{1}{\min(F,G)}} $.

\end{lemm}\label{lem:slowparaforspiral}

%%%%%%%%%%%%%%%%%%%%%%%
\begin{proof}
The functions  $h(x)=\frac{1}{x^F}f\circ s(x)$ and $k(x)=\frac{1}{x^G}g\circ \sigma(x)$ are slow since the functions $\frac{1}{x^F}$ and
$\frac{1}{x^G}$
are slow, as well as  $f\circ s$ and $g\circ \sigma$,  by Lemma \ref{majdkSinLogl}.
It is immediate that $h(x)$ and $k(x)$ satisfy condition
(\ref{eq: slow_+ parametrization}) of Remark \ref{rk: slow_+ parametrizations}.
Nevertheless, the following computation is provided in order to make explicit the constants attached to the slowness of $(h,k)$.  We have by formula (\ref{derprod}),
$$
 \frac{h^{(p)}(x)}{p!}
= \sum_{i=0}^p
\frac{1}{(p-i)!}(\frac{1}{x^F})^{(p-i)}  \frac{1}{i!}( f\circ s)^{(i)}(x).$$
Observe that
$$  \frac{1}{(p-i)!}(\frac{1}{x^F})^{(p-i)}
\le
\frac{F}{1}\frac{F+1}{2}\cdots \frac{F+p-i-1}{p-i}\frac{1}{x^{F+p-i}}\le \frac{1}{x^F}\frac{(F+1)^{p-i}}{x^{p-i}}.$$
It follows, by Lemma \ref{majdkSinLogl}, that for any $x\ge e$,  denoting again $A,B,C$ the constants attached to the slow function $s$,
$$\vert \frac{h^{(p)}(x)}{p!} \vert
\le
\sum_{i=0}^p \frac{1}{x^F}\frac{(F+1)^{p-i}}{x^{p-i}}
  \alpha^iA^i i^{(B+1) i} \frac{\log^{iC} x}{x^i}
 $$
$$
\le
\frac{1}{x^F}(p+1) (F+1)^p  \alpha^p  A^p p^{(B+1)p} \frac{\log^{pC} x}{x^p}. $$

Since for $p\ge 1$, one has for instance $p+1\le 2^p$, one finally obtains

$$\vert  \frac{h^{(p)}(x)}{p!}  \vert
\le
\frac{1}{x^F} \Big[ 2(F+1) \alpha A p^{B+1}
\frac{\log^{C} x}{x} \Big]^p, $$
and in the same way, denoting $a,b$ and $c$ the constants attached to the slow function $\sigma$,
$$\vert  \frac{k^{(p)}(x)}{p!} \vert
\le
\frac{1}{x^G}\Big[ 2(G+1) \alpha a p^{b+1}
\frac{\log^{c} x}{x} \Big]^p, $$
showing that $(h,k)$ satisfies  condition
(\ref{eq: slow_+ parametrization}) of Remark \ref{rk: slow_+ parametrizations}
for the following set of four constants
$$ 2\alpha(\max\{F,G\}+1)\max\{A,a\},  \ \max\{B,b\}+1,  \ \max\{C,c\}$$
$$  \hbox{ and } E= \max\{F,G\} .$$
In particular, by Lemma \ref{majdkSinLogl}, $(\phi_\ell, \psi_q)$   satisfies condition
(\ref{eq: slow_+ parametrization}) of Remark \ref{rk: slow_+ parametrizations}
for the constants
$$2\alpha(\max\{F,G\}+1)\max\{\ell,q\}, \ \max\{\ell,q\}+1, \max\{\ell,q\}-1$$
$$ \hbox{ and }    E= \max\{F,G\}.$$
By Remark \ref{rk:hcf2}(2), one can take $\varphi(T)=T^{\frac{1}{\min(F,G)}}$ for a height control
function.
\end{proof}
We now give an explicit possible value for the bound ${\cal B}(x,d)$
(de\-fi\-ned
in Theorem \ref{fewratpoints})   relative to the slow parametrized spiral of (\ref{eq:ParametrizationSpiral}).

%%%%%%%%%%%%%%%%%%%%%%%%%%%%%%%%%%%%%%%%%%%%%%%%%%%%%%%%%%%%%%%%%%%%%%%%%%%%%%%%%%%%%%%%%%%%%

We are searching  for a bound ${\cal B}(L,d)$ for the number of solutions of
$$P( \frac{1}{x^F} \sin  \circ \log^\ell x, \frac{1}{x^G} \cos \circ \log^q x) ,$$ for $x$ in some subinterval of $[1,+\infty[$ of length less
than $L$, and for $P\in \RR[X,Y]$ of degree less than $d$.
This amounts to bounding the number of solutions of
$$ Q(x,\sin\circ \log^\ell (x),\cos\circ \log^q(x))=0, $$
for $x$ in some subinterval of $[1,+\infty[$ of length less
than $L$, and for $Q\in \RR[X,Y,Z]$ of degree less than $d(F+G)$.
This finally amounts to bounding the number of solutions of the system
$$ Q(e^y,\sin(z),\cos(w))=z-y^\ell=w-y^q=0,$$
for $y$ in some subinterval of $[1,+\infty[$ of length less
than $\log L $, and for $Q\in \RR[X,Y,Z]$ of degree less than $d(F+G)$.
By the theorem in~\cite[\S1.4]{Kh91}, we obtain
$$ {\cal B}(L,d)\le 4d(F+G)\ell q(d(F+G)+\ell+q+2)^2
(\lfloor(\log L)/\pi)\rfloor+1).$$
Note that by Gwo{\'z}dziewicz et al.~\cite[Lemma 3]{GwKuPa99} or Benedetti and Risler~\cite[Lemma 4.2.6]{BeRi90}
we can dispose of the non-degeneracy hypothesis in the theorem of~\cite[\S1.4]{Kh91}
by bounding instead the number of solutions of the regular system
$$ Q(e^y,\sin(z),\cos(w))=\epsilon, z-y^\ell=w-y^q=0,$$ for $\epsilon$ a regular value
of $y\mapsto Q(e^y,\sin(y^\ell),\cos(y^q))$.

For an appropriate constant $\alpha'$, using the bound $\alpha' \log^4_+ T $ we just obtained for ${\cal B}(T^\frac{1}{\min(F,G)},\log T)$
and using the constants attached to the slow parametrization
$(\ref{eq:ParametrizationSpiral}) $ that we obtained at the end of the proof of  Lemma \ref{slowspiral},
 one sees by Theorem \ref{fewratpoints} and
Remark \ref{rk: slow_+ parametrizations}
that for $\varGamma$ the spiral parametrized by $(\ref{eq:ParametrizationSpiral})$, one can state
 the following proposition.

%%%%%%%%%%%%%%%%%%%%%%%%%%%%%%%%%%%%%%%%%%%%%%%%
\begin{prop}\label{fewratpointsLogSpiral} Let
$F,G>0$, $\ell,q\in \NN^*$ and $T\ge 1$. Then there exist
$N=N(F,G,\ell,q)$ (we can take $N=e^{\frac{\max\{\ell,q\}-1}{ \max(F,G)}}$) and
constants $\alpha=\alpha(F,G,\ell,q)$ and $\beta$  (we can take $\beta= 5+4\max\{\ell,q\}$) such that

$$\# \varGamma_{N,+\infty}(\QQ,T)\le \alpha
 \log_+^\beta T, $$
for $\varGamma$ the spiral parametrized by
$\gamma_{\ell,q}(x)=( \frac{1}{x^F}\sin \circ \log^\ell(x) ,
 \frac{1}{x^G}\cos \circ \log^q(x))$.
\end{prop}

%\begin{proof}[Sketch of proof of]
%The argument i
%\end{proof}

%%%%%%%%%%%%%%%%%
\begin{remark}\label{re:elementarycurve}
In Proposition \ref{fewratpointsLogSpiral} we could replace the spiral para\-me\-tri\-zed by $\gamma_{\ell,q}$ by a transcendental curve $\varGamma$ parametrized by
$$ x\mapsto (
u+\frac{1}{x^F} f\circ s(x),
v+\frac{1}{x^G}g \circ \sigma(x)),
$$
where $F,G>0$,
$u,v\in \RR$ and
\begin{itemize}
\item $f$ and $g$ are in the set of functions $h$ satisfying
$$h=\operatorname{Id} \hbox{ or } \exists \alpha (\ge 1), \
\forall x\ge d, \ \forall p\ge 0,\ \vert h^{(p)}(x)\vert \le \alpha^p.$$
\item  $f$, $g$, $s$ and $\sigma$ are elementary functions in the sense of \cite[\S1.5]{Kh91} (defined from the simple functions $e^x$, $\sin x$, $\cos x$, $\log x$, $\arcsin x $, $\arccos x $, $\tan x $, $\arctan x$ and rational functions,
by induction using composition),
\item $s$ and $\sigma$ are compositions of slow functions, respectively with $\log^\ell$ and  $\log^q$, for some $\ell, q\in \NN^*$,
\item   one of the following conditions on $u$ and $v$
is satisfied
\begin{enumerate}
\item $u\in \QQ$ and $f$ has no zeros,
\item $v\in \QQ$ and $g$ has no zeros,
\item $u$ and $v$ are both rational,
\item $u\not \in \QQ$ and $u$ is not a
$U$-number of degree $\nu=1$ in Mahler's classification,
\item $v\not \in \QQ$ and $v$ is not a
$U$-number of degree $\nu=1$ in Mahler's classification.
\end{enumerate}

\end{itemize}
In this situation, on one hand this parametrization is slow by Lemma \ref{slowspiral}, and on the other hand, by the theorem of \cite[\S1.6]{Kh91},  ${\cal B}(L,d)$ is  polynomially bounded in $d$
and $\log_+ L$.
Furthermore, by Remarks \ref{rk:hcf2},
%and  \ref{rk: slow_+ parametrizations},
any of the conditions $(1),(2),(3)$ or $(4)$ on
$u$ and $v$
ensure that one can take a power of $T$ as a height control function for $\gamma$.
(In order to apply Remark \ref{rk:hcf2}$(2)$ to condition $(3)$, observe that every common zero of $f\circ s$ and $g\circ \sigma$ maps to the single point $(u,v)$.)
%\changeGC{ Note that, assuming condition $(3)$ above, Remark \ref{rk:hcf2}$(2)$ requires that $f\circ s$ and $g\circ \sigma$ have
%no common zero in order to have for a height control function a power of $T$, say $T^M$ (where $M$ is an explicit function of $F,G$ and the height of $(u,v)$).
%However, if $x$ is a parameter for a
%common zero of $f\circ s$ and $g\circ \sigma$, then the point $(x,\gamma(x)) $
%
%$\Gamma$ passes at
%the rational point $(u,v)$ at $x$. Hence parameters of rational points of height at most $T$, except those of the only special point $(u,v)$, are bounded by $T^M$. Thus, for counting rational points,
%one can ignore in Remark  \ref{rk:hcf2}(2) the requirement about common zeroes of  $f\circ s$ and $g\circ \sigma$. }
%%
%%
 In conclusion, our assumptions on $f,g,s,\sigma, u$ and $v$  imply, in the same way as for Pro\-po\-si\-tion \ref{fewratpointsLogSpiral}, that $\varGamma$ has finite order.

Note that functions $f,g,s$ and $\sigma$ satisfying the above conditions can be built using Remarks \ref{re:AlgeraSlow}.
For instance, $f$ and $g$ can be built from $\sin$, $\cos$,
$\arctan$ and rational functions of negative degree (which are bounded elementary functions, and composition
of those functions  with elementary functions), in order to get elementary bounded functions.
Functions in the  algebra generated by bounded rational functions  in $\log^{r}$,
$\sin \circ \log^\ell$, $\cos \circ \log^q$ and $\arctan  \circ \log^m$ are instances
of slow functions $s$ and $\sigma$.

\end{remark}

\begin{proof}[Proof of Propositions \ref{introcurve} and \ref{introspiral}]
The proof of Propositions \ref{introcurve} and \ref{introspiral} are straightforward,
after repar\-ame\-tri\-zation by $\log$, since the assumptions of Proposition \ref{introcurve} are an axiomatization of the proof of Proposition \ref{fewratpointsLogSpiral}, as made in Remark \ref{re:elementarycurve}. Thus  Remark \ref{re:elementarycurve}  provides a compact interval $J$ of the parameter outside which the curve of Proposition \ref{introcurve} has finite order. But for the piece of this curve parametrized by $J$,
one can apply again our computation, which reduces in the compact case exactly to the computation of \cite{Pi2006} through mild parametrizations.
\end{proof}

%\begin{remark}\label{re: spiral around non rational centre}
%A translation of the spiral of Proposition \ref{fewratpointsLogSpiral}, or more generally of a curve like in Remark \ref{re:elementarycurve}, by an algebraic number or by a transcendental number having a transcendence measure function $\tau$ bounded by $M\log(T)$, for some $M>0$ has also  a finite order, since by Remark \ref{re: on the limit point of the curve 2} (4), in this case one can take for height control function $\varphi(T)$ of $\gamma$ a power of $T$ and for ${\cal B}(L,d)$ a polynomial in $d$ and in $\log(L)$.
%Examples of such numbers are $\pi$, $e^\pi$ or $\log v$, for any algebraic $v$ other than $0$ or $1$, and all $S$-numbers.  Recall that almost all (in the sense of Lebesgue measure) real numbers are S-numbers.  See \cite[Chapters 3 and 8]{Ba75} for details.
%\end{remark}

%%%%%%%%%%%%%%%%%%%%%%%%%%%%%%%%%%%%%%%%%%%%%%%%%%%%%%%%%%%%%%%%%%%%%%%%%%%%%%%%%%%%%%%%%%%%%%%%%%
\subsection{The case of graphs.}\label{Section:TheCaseofGraphs}

We proceed here similarly as in the preceding section to establish Proposition \ref{introeuler} as a corollary of a detailed proof of a special case.

Given a function $g \colon  J\to \RR $ on some interval $J$ of $\RR$, we denote by
$\varGamma$ (or if needed $\varGamma_g$) the graph of $g$, and for $I\subset J$,
 we denote by
$\varGamma_I$ (or $\varGamma_{g,I}$) the set $\varGamma\cap (I\times \RR)$.

We begin by noting that the case of  $\varGamma$ (for $g$ having  controlled decay) is encompassed by the discussion of Section \ref{Scounting}.
Indeed, for $g  \colon [1,+\infty[\to \RR$ satisfying the conditions
of Definition \ref{slowp}, the map $\gamma \colon  [1,+\infty[\to \RR^2$
defined by $\gamma(x)=(\frac{1}{x}, g(x))$ is a slow parametrization
of a curve of $\RR^2$, with rational points
  in bijective correspondence with rational points of same height
 in   $\varGamma$.
Moreover, $\gamma$ satisfies condition (\ref{eq: slow_+ parametrization}) of Remark \ref{rk: slow_+ parametrizations}, and so as a consequence of Remark \ref{rk: slow_+ parametrizations} and  Theorem \ref{fewratpoints}, we can state the following proposition.

%%%%%%%%%%%%%%%%%%%%%%%%%%%%%
  \begin{prop}\label{fewratpointsgraph}
  Let  $g \colon [a,+\infty[\to \RR$ be a slow function with constants $A$, $B$ and $C$.
   Assume that there exists some function $\varphi  \colon [1,+\infty[\to \RR$ such that  the height
of any rational point of $\varGamma_{\varphi(T),+\infty } $ is
$\ge T$.  Then there exist $N=N(C)$ (given by Proposition \ref{BoundAboveGeneral} and that can be for instance
$e^C$) and a constant $\alpha=\alpha(A,B,C)$
such that
$$\# \varGamma_{N,+\infty}(\QQ,T)\le \alpha
 \log_+^{2(B+C)} T
\log(\varphi(T)){\cal B}(\varphi(T),\log_+ T).$$
In particular, since one can always take $\varphi(T)=T$ by Remark \ref{rk:hcf3}(3),
$$\# \varGamma_{N,+\infty}(\QQ,T)\le \alpha
 \log_+^{2(B+C)+1} T
{\cal B}(T,\log T). $$
Hence $\varGamma$ has finite order as soon as there exists a polynomial $Q$ such that
${\cal B}(x,d)\le Q(\log x,d)$.
\end{prop}

%%%%%%%%%%%%%%%%%%%%%%%%%%%%%%%%ù
\begin{remark}\label{RemarkSincx}

By the Lindemann-Weierstass theorem there are no rational points in the graph of $\sin$ (resp. $\log$) except the point $(0,0)$ (resp.\ $(1,0)$), since for a nonzero
algebraic number $x\in \CC$, $e^{x}$ is transcendental. A natural way to build functions from $\sin$ with graph having {\sl a priori} the most chance to contain rational points is to compose $\sin$ with a function sending
rational points to transcendental ones, such as $x\mapsto r  x$,  for $r$ a transcendental number or $x\mapsto \log^\ell x$, for $\ell\in \NN^*$. Proposition \ref{fewratpointsgraph} allows us to treat both cases.

In the first case, for the function $x\mapsto  \sin(cx)$, $c\in \RR_+^*$,  restricted to a compact interval $[a-\frac{\pi}{2c},a+\frac{\pi}{2c}]$, where $a\in \RR$, the method of proofs of Section \ref{Scounting} applies, and reduces to the methods of \cite{BoPi1989} and \cite{Pi2006} for analytic functions defined on compact intervals. In this situation, since \cite[\S1.4]{Kh91} provides $\alpha' \log^2 T $ as bound for the number of solutions of $P(x,\sin(cx))=0$, $x\in [a-\frac{\pi}{2c},a+\frac{\pi}{2c}] $, $\deg(P)=\lfloor\log_+ T\rfloor$, we get a bound
$\#\varGamma_{[a-\frac{\pi}{2c},a+\frac{\pi}{2c}]}\le \alpha'' \log^2_+ T. $ Here the constant $\alpha''$ does not depend on
$a$, since we have uniform bounds for the derivatives of $x\mapsto \sin(cx)$ with respect to $a$. Consequently for the graph $\varGamma_c$ of $x\mapsto \sin(c x)$, we have
$$\#\varGamma_{c,\RR} (\QQ,T)\le \alpha c T\log_+^2 T. $$
This bound is quite sharp, since
$\#\varGamma_{\frac{\pi}{n},\RR}(\QQ,T)$ is bounded from below by $ \frac{\alpha'''}{n} T$ for $T\ge n$.

The second case cannot be reduced to the compact case and uniform bounds for derivatives with respect to translations, and thus requires control on the derivatives at infinity, as in the assumption of Proposition  \ref{fewratpointsgraph}. We hereafter treat this case as a consequence of Proposition \ref{fewratpointsgraph}.

\end{remark}

  %%%%%%%%%%%%%%%%%%%%%%%%%%%%%%%%%%%%
  \begin{cor}\label{Cor:fewgraphsinlogl} Let  $\ell \in \NN^*$, $g_\ell \colon \RR_+^*\to \RR$  be the function defined by
  $g_\ell(x)= \sin \circ \log^\ell(x)$ and let $\varGamma_\ell$ its graph. Then there exist constants $\alpha=\alpha_\ell$ and $\beta=\beta_\ell$  ($\beta=5+4\ell$ being possible) such that for any $T\ge 1$,
  $$ \varGamma_\ell(\QQ,T)\le \alpha  \log_+^\beta T. $$

  \end{cor}

  %%%%%%%%%%%%%%%%%%%%%%%%%%%%%%%%%%%%%%%
  \begin{proof} Using the theorem of \cite[\S1.4]{Kh91} in the same way that we did
  for Proposition \ref{fewratpointsLogSpiral}, one obtains here, for the curve $\varGamma_{\ell} \cap ([1,T]\times \RR)$, the bound
  ${\cal B}(T,d)\le  4d\ell(d+\ell+ 2)^2
(\lfloor\frac{\log T}{\pi}\rfloor+1)$.
  Since by Lemma \ref{slowspiral} the derivatives of  $g_\ell$ satisfy  the bound required by Proposition \ref{fewratpointsgraph}, one deduces from this proposition the existence of constants
  $N,\alpha$ and $\beta$, depending only on $\ell$, such that
  $$\# \varGamma_{\ell, [N,+\infty[}(\QQ,T)\le \alpha
  \log_+^\beta T. $$
Since ${\cal B}(T,\log T)\le \alpha'_\ell
\log^4 T$, for some $\alpha'_\ell>0$, by Lemma \ref{majdkSinLogl} and by Proposition \ref{fewratpointsgraph}, $\beta=5+4\ell$ is possible.
Assuming that $N\ge 1$,  a bound on the same kind  also holds over the interval  $]0,1/N]$ since the one-to-one transformation $(x,y)\mapsto(1/x,-y)$ maps
the rational points of height less than $T$ of  $\varGamma_{\ell, ]0;1/N ] }$ onto the rational points of
height at most  $T$ of $\varGamma_{\ell, [N ;+\infty )}$.
Finally, over $[1/N,N]$, by \cite{Pi2006}, we also have the same kind of bound for $\# \varGamma_{\ell,[1/N,N]}(\QQ,T)$
since $g_{\ell\vert [1/N,N]}$ is   an analytic function on a compact domain and thus this graph
comes with its obvious mild-parametrization (see the discussion after Theorem 1.5 in \cite{Pi2006}).

  \end{proof}

%\changeGC{
%\begin{remark}
%We recall that the structure $(\RR,+,.,\Gamma_\ell)$ defines $\ZZ$ if and only if $\ell\not=1$.
%But since our bound for $\#\Gamma_\ell(\QQ,T)$ requires on one hand bounds for the derivatives of $g_\ell$ and on the other hand a
%B\'ezout bound
%for the number of solution of $P(x,g_\ell(x))=0$ on some interval, the continuity with respect to $\ell$ of those two last bounds
%forces continuity in our bound for  $\#\Gamma_\ell(\QQ,T)$. In consequence we cannot detect the
%logico-arithmetical bifurcation at $\ell=1$ here. Moreover,  since we have no estimate on the sharpness of our bound for $\#\Gamma_\ell(\QQ,T)$,
%we still do not know whether or not such a logico-arithmetical bifurcation can be detected by a singularity of the density function
%of rational points in $\Gamma_\ell$.

%\end{remark}

\begin{remark}\label{re:elementarygraph}
More generally, and similarly to Remark \ref{re:elementarycurve} and the proofs of Propositions \ref{introcurve} and \ref{introspiral}, one can consider transcendental
elementary functions defined in \cite{Kh91} that are given by composition of a slow function with some power of $\log$ (to have a B\'ezout bound ${\cal B}(x,d)$ as required in Proposition \ref{fewratpointsgraph}), in order to get instances of graphs with finite order.
For instance, the graph of any coordinate of the curve of Example \ref{exampleintro} has finite order.  In this way we get in particular a proof of Proposition \ref{introeuler}.
Concerning
Proposition \ref{introeuler} as stated in the introduction, note that the function $af(c\log^\ell)$ might not bounded by $1$ in absolute value (as required in the definition of slow function), but then $af(c\log^\ell)/\lfloor a +1 \rfloor$ is slow, and finally if $\alpha\log^\beta T$ bounds
$\#\varGamma_{\frac{af(c\log^\ell)}{\lfloor a +1 \rfloor}}(\QQ,T)$ for some $\alpha, \beta$, then
$\alpha'\log^\beta T$ bounds
$\#\varGamma_{af(c\log^\ell)}(\QQ,T)$, for some $\alpha'$.

\end{remark}

\begin{remark}
Similar results stating that the graph of some function has finite order have been recently proved, in particular for entire functions from $\CC$ to $\CC$ (see \cite{BoJo15b}), where B\'ezout estimates are provided by the growth of these functions in the spirit of Coman and Poletsky~\cite{CoPo07}.  In the real case, since the growth of the derivatives is not prescribed by the growth of the function itself, one has to consider some
bound for all the derivatives as an assumption. For instances of functions with graphs of finite order (over a compact interval),
 see \cite{CM2016}, where indeed the assumptions concern the Taylor coefficients at some point of the series.
\end{remark}

\begin{remark}
When it comes to counting rational points on graphs, a classical function to look at is the Riemann zeta function $\zeta \colon ]1,+\infty[\to \RR$, given by $\zeta(x)=\sum_{n=1}^{+\infty}\frac{1}{n^x}$. Let us denote its graph by $\varGamma_\zeta$. By van den Dries and Speissegger~\cite{DrSp00}, $\varGamma_\zeta$ is o-minimal, so $\#\varGamma_\zeta(\QQ,T) $ is sub-polynomial by \cite{PiWi2006}.
Moreover, it  is known since \cite{Ma11} that for some constant $\alpha>0$
$$\#\varGamma_{\zeta, ]2,3[}(\QQ,T)\le
\alpha \displaystyle \frac{\log^2 T }{(\log \log T)^2}.  $$
The interval $]2,3[$ may be replaced by any  bounded interval, as proved in \cite{Be11}. In \cite{BoJo15} it is finally proved that one can bound $\#\varGamma_{\zeta, ]1,+\infty[}(\QQ,T)$ (as well as the number of algebraic points of  height $\le T$ and degree $\le k$ over $\QQ$) in the following way: for some constant $\alpha>0$,
$$\#\varGamma_{\zeta, ]1,+\infty[}(\QQ,T) \le \alpha \log^3 (T)\log^3 \log T.$$
It is indicated in \cite[ page 1154]{BoJo15} that one can even get a
$\log^2(T)\log^2 \log T $ bound.
We can here easily give a bound in the form $\log^4(T)\log \log T $ as a consequence of Proposition \ref{fewratpointsgraph}.

Another classical special function that can be treated by our approach is the Euler $\Gamma$ function defined by
$\displaystyle \Gamma(x)=\int_0^{+\infty} t^{x-1}e^{-t} \ dt$, considered for $x\ge 1$.
Let $\varGamma_\Gamma$ be the graph of this function.
 As for the Riemann zeta function,
$\varGamma_\Gamma$ is o-minimal (again see \cite{DrSp00}).
It has been proved in \cite{Be14} that for any interval $I$ of length $1$,
$$\displaystyle \#\varGamma_{\Gamma,I}(\QQ,T) \le \alpha \frac{\log^2 T }{\log \log T}, $$
and in \cite{BoJo15} the following bound is given:
$$ \#\varGamma_{\Gamma,]0,+\infty[} \le \alpha \log^2(T)\log^3 \log T .$$
We give below a very rough bound of the form $\log^{11}(T)\log \log T$ as a direct application of Proposition \ref{fewratpointsgraph}.

 For the special functions $\zeta $ and $\Gamma$
our bounds have no better exponents then the best known exponents. But be aware that no rational points are expected, or at least known, in the graphs of these functions, except for $(n,n!)$, $n\in \NN^*$, for $\Gamma$.
Thus, even a small bound of type $\alpha\log^{\beta} T$ is probably very far from being optimal.
 Note that by \cite[Theorem 8.2]{CoPo07}, the B\'ezout estimate that we shall use for $\zeta$ is quite sharp; this just shows that any existing general method based on B\'ezout estimate
probably will not produce particularly sharp bounds.
We produce bounds here only to show that our method works for these particular instances
of functions, viewed as real ones,
 and how it can be adapted for the Euler function, since this function  is not slow.
 Toward this end, we try to present shorter computations that are probably not the sharpest that one can get.

\subsubsection{The Riemann $\zeta$ function}\label{RiemanZeta}
First of all, we observe that the derivatives of $\zeta$ satisfy the bound of Proposition \ref{fewratpointsgraph} on any interval
$[a,+\infty[\subset ]1,+\infty[$, for some $A=A_a, B=1$ and $C=0$.
Indeed, we have on one hand for $p\ge 0$, $\zeta^{(p)}=\displaystyle \sum_{n\ge 1} (-1)^{p}\frac{\log^p n}{n^x}$.
On the other hand, since the study of $\displaystyle x\mapsto \frac{x^p}{n^x}$ shows that
$\displaystyle \frac{1}{n^x}\le \frac{p^p}{e^px^p \log^p n} $, for any $a>1$ one can choose
$\lambda =\frac{1}{2} -\frac{1}{2a}\in ]0,1[$,
such that for any $x\ge a$ and $p\ge 1$,
$$\displaystyle \vert \zeta^{(p)}(x)\vert
\le  \sum_{n\ge 1}  \frac{\log^p n}{n^{\lambda x}} \frac{1}{n^{(1-\lambda) x}}
\le   \frac{p^p}{(\lambda e)^px^p}  \zeta((1-\lambda)x)$$
$$
\le  \frac{p^p}{(\lambda e)^px^p}
\zeta(\frac{a}{2}+\frac{1}{2})\le
\Big(\frac{ \zeta(\frac{a}{2}+\frac{1}{2})}{\lambda e}\Big)^p \frac{p^p}{x^p}.$$
Then observe that, with the same $\lambda$ as above
and the notation
$m_a=\zeta(\frac{a}{2}+\frac{1}{2})-1$, for $\zeta$ one can take for a height function control
\begin{equation}\label{hcfzeta}
\varphi(T)=\displaystyle\frac{\log(m_aT)}{\lambda\log 2}.
\end{equation}
This follows from the following remarks.
In case $\zeta(x)\in \QQ$ as height less than $T$, $T\ge 1$, then
$\zeta(x)-1 \ge \frac{1}{T}$.
But for any $x\in [a,+\infty[$,  one has
$$ \zeta(x)-1 \le \sum_{n\ge 2}\frac{1}{n^{\lambda x}}\frac{1}{n^{(1-\lambda) x}} \le   m_a2^{-\lambda x}.  $$
 Consequently, for $x\ge \frac{\log(m_aT)}{\lambda\log 2}$, $\zeta(x)$ has height at least $T$.

Finally, by \cite[Theorem 8.2]{CoPo07} or \cite[Proposition 1]{Ma11}, one knows that for some constant $c>0$,
\begin{equation}\label{Bezoutzeta}
{\cal B}(\varphi(T),\log T)\le
c(\log(T)+\varphi(T)\log \varphi(T))\log T.
\end{equation}
From  (\ref{hcfzeta}), (\ref{Bezoutzeta}) and Proposition \ref{fewratpointsgraph} we deduce that for any $a>1$ and
$T\ge 3$,
$$ \varGamma_{\zeta, [a,+\infty[}(\QQ,T)\le \alpha \log^4(T)\log \log T.$$

\begin{remark}\label{rk: bounding zeta by 1}
For $a>1$, on $[a,+\infty[$, $\zeta$ may be not bounded by $1$, and thus on this
interval $\zeta $ is not a slow function. But as already noticed
in Remark \ref{re: f and g bounded by 1}, one can always divide $\zeta$ by some large enough integer $M_a$ in order to fulfil the definition of slow function, since $\alpha' \log^\beta T$ bounds $\#\varGamma_{\zeta, [a,+\infty[}(\QQ,T)$ whenever $\alpha \log^\beta T$ bounds $\#\varGamma_{\frac{\zeta}{M_a}, [a,+\infty[}(\QQ,T)$.
\end{remark}

\begin{remark}
For $u,v$ two real numbers such that $u+v$ is irrational and not a $U$-number of degree $1$, one deduces from the study above that the graph of $u+v\zeta$ has finite order on $[a,+\infty[$. Indeed, up to dividing by a large enough integer as observed in Remark \ref{rk: bounding zeta by 1} , one has
that $u+v\zeta$ is slow. Furthermore a B\'ezout bound for $\zeta$
is also a B\'ezout bound for $u+v\zeta$. Finally, taking into account the form of the B\'ezout bound for $\zeta$ given in
(\ref{Bezoutzeta}), to apply Proposition \ref{fewratpointsgraph} it remains to prove that some power of $\log T$
is a height control function for the graph of $u+v\zeta$.
For this, observe that on one hand
$$ \vert u+v\zeta(x) -(u+v) \vert \le \vert v\vert m_a 2^{-\lambda x}, $$
and on the other hand, since an irrationality measure function for
$u+v$ is of the form $K\log T$ ($K>0$), whenever $ u+v\zeta(x)$ is a rational number of height $\le T$, one has
$$ \frac{1}{T^K}  \le \vert u+v\zeta(x) -(u+v) \vert  .$$
It follows that if $ u+v\zeta(x)$ is a rational number of height $\le T$, then $ \frac{1}{T^K}  \le \vert v\vert m_a 2^{-\lambda x}$,
and so $x\le K'\log T$ for some $K'>0$.

\end{remark}

%**** the branch $\Gamma_{]1,a]}$ may be treated in the same way, since $\zeta$ has o pole of order 1 at 1. thus we have the result on $]0,+\infty[$ ? $\sin(x)\zeta(x)$ is slow, since $\sin$ and $\zeta$ are. Do we have a C\&P bound for $\sin(x)\zeta(x)$ ?*****

%%%%%%%%%%%%%%%%%%%%%%%%%%%%%%%%%%%%%%%%%%%%%%%%%%%%%%%%%%%%%%%%%

\subsubsection{The Euler $\Gamma$ function}\label{EulerGamma}
We first remark that since for any $p\ge 0$, since  $\displaystyle\frac{\log^p t}{t}\le
(\frac{p}{e})^p $
 and since $\Gamma^{(p)}(x)=\displaystyle \int_0^{+\infty} \log^p(t)t^{x-1}e^{-t}\ dt$, one has
\begin{equation}\label{eq:BoundDerGamma}
 \Gamma^{(p)}(x)\le (\frac{p}{e})^p\Gamma(x+1)=
(\frac{p}{e})^p x \Gamma(x).
\end{equation}
Now let us denote by $f$ the inverse function of $\Gamma$ on $[1,+\infty[$, and $x=f(y)$. One can show by induction on $p\ge 1$ that
$\displaystyle \frac{f^{(p)}(y)}{p!}$ is a sum of at most $p!$ terms   of the form
$$ c (\Gamma^{(j_1)}(x))^{m_1}\cdots (\Gamma^{(j_p)}(x))^{m_p}
(\Gamma'(x))^{-k}, $$
with $\vert c \vert \le 2^{p-2}$,
$k\in [0,2p-1]$,
$ j_1,\cdots, j_p \in [2,p]$,
$ m_1+ \cdots + m_p\in [0,p-1]$,
$-k+j_1m_1+\cdots + j_pm_p=-1$ and
$-k+ m_1+\cdots +m_p=-p.$
From this observation and from (\ref{eq:BoundDerGamma}) one has for any $p\ge 1$ and any $y\ge 1$
and for a set $J$ of indices $m_i, j_r$ of cardinality less than $p!$,
$$ \frac{f^{(p)}(y)}{p!}
\le 2^p\sum_J (\frac{p}{e})^{ \sum_{r=1}^pj_rm_r}
 (x\Gamma(x))^{ \sum_{r=1}^p m_r}
 (\Gamma'(x))^{-k}$$
$$ \le (\frac{2}{e^2})^p p^{2p}  \sum_J (x\Gamma(x))^{ \sum_{r=1}^p m_r}
 (\Gamma(x))^{-k}$$
 $$
\le p!(\frac{2}{e^2})^p  p^{2p}
 (\frac{x}{\Gamma(x)})^p=  (\frac{2}{e^2})^p  p^{3p}
 (\frac{f(y)}{y})^p. $$
Since for some constant $D$, for any $x\ge 1$, $\Gamma(x)\ge D e^x$ we have for some constant $\delta >0$, for any $y\ge 1$, $\log y\ge \delta f(y)$, and thus one has for
any  $y,p\ge 1$
\begin{equation}\label{eq:SlowGamma}
\vert  \frac{f^{(p)}(y)}{p!}\vert \le
(\frac{2}{\delta e^2})^p  p^{3p} (\frac{\log y}{y})^p.
\end{equation}
This does not show that $f$ is slow, since $f$ is not bounded. But, as already noted in Remark \ref{1/f},
a direct computation using $f(y)\ge 1$, formula (\ref{FaaDiBruno}) and the inequality (\ref{eq:SlowGamma})
shows that $1/f$ is slow, with constants
 $ A= \displaystyle\frac{4}{\delta e^2}, B=4, C=1.$
 Note that
$$\#\varGamma_{\frac{1}{f}, [1,T]}(\QQ,T)
=\#\varGamma_{f, [1,T]}(\QQ,T)$$
$$
=  \#\varGamma_{\Gamma, [1,\Gamma^{-1}(T)]}(\QQ,T)
= \#\varGamma_{\Gamma, [1,+\infty[}(\QQ,T).$$
 In order to bound $\#\varGamma_{\frac{1}{f}, [1,T]}(\QQ,T)$ using Proposition \ref{fewratpointsgraph}, we need to produce for
 $1/f$ a bound
 $b_{1/f}(T,\log T)$. But clearly it is enough
to find a  bound
 $b_\Gamma(\frac{\log T}{\delta},\log T)$ for $\Gamma$, since again,
 $\log y\ge \delta f(y)$.  Such a bound is provided by \cite[Proposition 3.1]{Be14} in the form
 $$b_\Gamma(\frac{\log T}{\delta},\log T)\le c \log^2(T)\log \log T.$$
 We conclude by Proposition \ref{fewratpointsgraph} that
 $$ \#\varGamma_{\Gamma, [1,+\infty[}(\QQ,T)\le \alpha \log^{11}(T)\log \log T. $$

\end{remark}

%%%%%%%%%%%%%%%%%%%%%%%%%%%%%%%%%%%%%%%%%%%%%%%%%%%%%%%%%%%%%%%%%%%%%%%%%%%%%%%%%%%%%%%%%%%%%%%%%
\section{Some connections to logic}\label{S;logic}

In this final section, the reader is assumed to be familiar with
definability theory over the field of real numbers (see, e.g.,
\cite{DrMi96} or Wilkie~\cite{Wi09} for a brief
introduction).

Given $E\subseteq \RR^n$, we let $E^{\mathrm{trans}}$ be the
result of removing from $E$ all infinite semialgebraic subsets.
(Hence, the only nonempty semialgebraic subsets of $E^{\mathrm{trans}}$ are singletons.)

Let $\mathfrak{R}$ be a fixed, but arbitrary, structure on the
real field; ``definable'' means ``definable with parameters'', unless indicated otherwise.

The seminal paper~\cite{PiWi2006} established the possibility
of obtaining uniform large-scale asymptotics on height bounds
of definable sets.
We recall the basic result:

\begin{thm}[{\cite[1.8]{PiWi2006}}]\label{pwthm}
If $\mathfrak{R}$ is o-minimal and $E$ is definable, then $\#E^{\mathrm{trans}}(\QQ,T)$ is sub-polynomial.
%
%$\log
%(\#E^{\mathrm{trans}}(\QQ,T))=o(\log T)$ for every definable
%set~$E$.
\end{thm}

This is the best possible bound in this
generality (see \cite{Pi06counterexample}, \cite{Su02},  \cite{Su06}  for information). Two questions arise
naturally:
\begin{itemize}
\item[(A)] To what extent is o-minimality necessary?
\item[(B)] Are there examples of $\mathfrak{R}$ such
that $E^{\mathrm{trans}}$ has finite order for every definable
set $E$?
\end{itemize}

There is a trivial positive answer to~(B), because $E^{\mathrm{trans}}$
is finite if and only if $E$ is semialgebraic, if and only if $E$
is definable in the real field.
Thus, we modify the question:
\begin{itemize}
\item[(B${}'$)] Are there examples of $\mathfrak{R}$ such that
$E^{\mathrm{trans}}$ has finite order for every definable set
$E$, and there is some definable $S$ such that
$S^{\mathrm{trans}}$ contains a compact set of positive
topological dimension?
\end{itemize}
By Binyamini and Novikov~\cite{BiNo16},
$\RR^{\mathrm {RE}}$ (the expansion of the real field by the
restrictions of $\exp$ and $\sin$ to $[0,1]$) provides a
positive answer to (B${}'$). 
It is well known by now that
$\RR^{\mathrm {RE}}$ is o-minimal. 
Wilkie has conjectured that
the expansion of the real field by $\exp$ (on all of $\RR$) is
another example; for more information on Wilkie's Conjecture
and progress theretoward,
see~\cite{BiNo16}, \cite{ClPiWi16}, \cite{JoMiTh11}, \cite{JoTh12},  \cite{Pi2006}, \cite{Pi07}, \cite{Pi10}.
We have a conjecture of our own:

\subsubsection*{Conjecture}
$E^{\mathrm{trans}}$ has finite order for each $E$ definable in
the expansion of $\RR^{\mathrm {RE}}$ by any logarithmic spiral
$\mathbb S_\omega$.
(We regard our result that $\mathbb
S_\omega$ has finite order as an encouraging first step.)
A strictly (by Tychonievich~\cite{Ty12}) weaker version:
$E^{\mathrm{trans}}$ has finite order for each $E$ definable in
$(\RR,+,\cdot,\mathbb S_\omega)$. 
Even weaker (potentially):
$E^{\mathrm{trans}}$ has finite order for each $E$
$\emptyset$-definable in $(\RR,+,\cdot,\mathbb S_\omega)$.

As for question~(A), there is an obvious necessary condition:
If $\mathfrak R$ defines the set of integers, $\ZZ$, then we
cannot have even $\#E^{\mathrm{trans}}(\QQ,T)=o(T)$ for every
$\emptyset$-definable $E\subseteq \RR$. 
We can do better:

\begin{prop}\label{nwd}
The following are equivalent.
\begin{enumerate}
\item
$\#E^{\mathrm{trans}}(\QQ,T)=o(T)$ for every definable
$E\subseteq \RR$.
\item
Every definable subset of $\RR$ either
has interior or is nowhere dense.
\item
%$\log(\#E(\QQ,T))=o(\log T)$
For every definable $E\subseteq \RR^n$, if
no coordinate projection of $E$ has interior, then
$\#E(\QQ,T)$ is sub-polynomial.
\end{enumerate}
(And similarly with ``definable'' replaced by
``$\emptyset$-definable''.)
\end{prop}

\begin{proof}
1$\Rightarrow$2.
Let $E\subseteq \RR$ be definable and have no
interior.
Suppose to the contrary that $E$ is dense in some
nonempty open interval $I$; 
then $(I\cap
E)^{\mathrm{trans}}=I\cap E$ and $(I\setminus
E)^{\mathrm{trans}}=I\setminus E$, yielding
$\#I(\QQ,T)=\#(I\cap E)(\QQ,T)+\#(I\setminus E)(\QQ,T)=o(T)$,
which is clearly false.

2$\Rightarrow$3.
By Hieronymi and Miller~\cite[1.4]{HiMi15},
$E$ has Assouad dimension zero (see~\cite[\S~4]{HiMi15} for the
definition). 
Thus, given $\epsilon>0$ there exists $C>0$ such
that for all $T>1$
$$
\#E(\QQ,T)\leq \operatorname{net}_{1/T}(E\cap [-T,T]^n)\leq
C\left(\frac{T}{1/T}\right)^{\epsilon/2}=CT^\epsilon.
$$
%Hence, $\log (\#E(\QQ,T))=o(\log T)$.

3$\Rightarrow$1.
If $E\subseteq \RR$, then
$E^{\mathrm{trans}}\subseteq E\setminus\operatorname{int}(E)$;
if $E$ is definable, then so is
$E\setminus\operatorname{int}(E)$.
\end{proof}

There are several classes of structures known to satisfy
condition~\ref{nwd}.2 that are not o-minimal. 
It would take us
too far afield to discuss them here, but see~\cite{HiMi15}
and~\cite{Mi05} for examples and references.
Remarkably, the following questions seem to be open:
Does Condition~\ref{nwd}.2 imply that $\#E^{\mathrm{trans}}(\QQ,T)$ is sub-polynomial
%$\log
%(\#E^{\mathrm{trans}}(\QQ,T))=o(\log T)$
for every definable set $E$?
If $E\subseteq \RR^n$ is a boolean combination of open sets and
$(\mathbb{R},+,\cdot,E)$ does not define $\ZZ$, is $\#E^{\mathrm{trans}}(\QQ,T)$ sub-polynomial?

\bibliographystyle{acm} %acm-fr : format acm modifié par moi pour remplacer "AND" par "ET" dans la liste des auteurs / unsrt pour l'orde des références donné par l'ordre des
% citations dans le fichier .bib /  style plain, acm etc... pour l'ordre alphabétique des auteurs
\bibliography{biblio}
%\bibliographystyle{acm} %unsrt pour l'orde des références donné par l'ordre des
% citations dans le texte: style plain, acm etc... pour l'ordre alphabétique des auteurs
%%%%%%%%%%%%%%%%%%%%%%%%%%%%%%%%%%%%%%%%%%%%%%%%%%%%%%%%%%%%%%%%%%%%%%%%%%%%%%%%%%%%%

\end{document}